\documentclass{amsart}
\pdfoutput=1
\usepackage{amsmath,amsthm,amssymb}
\usepackage{mathrsfs}
\usepackage{hyperref}
\usepackage{graphicx}
\usepackage{color}

\usepackage[all]{xy}
\SelectTips{cm}{10}

\usepackage[small,nohug,heads=LaTeX,midshaft]{diagrams}
\diagramstyle[labelstyle=\scriptstyle]
\newarrow{Inline}C----
\newarrow{Corresponds}<--->

\numberwithin{figure}{section}
\numberwithin{equation}{section}

\theoremstyle{plain}
\newtheorem{thm}{Theorem}[section]
\newtheorem{prop}[thm]{Proposition}
\newtheorem{lem}[thm]{Lemma}
\newtheorem{cor}[thm]{Corollary}
\newtheorem{conj}[thm]{Conjecture}

\theoremstyle{definition}
\newtheorem{defn}[thm]{Definition}
\newtheorem{example}[thm]{Example}

\theoremstyle{remark}
\newtheorem{rem}[thm]{Remark}

\newcommand{\la}{\langle}
\newcommand{\ra}{\rangle}
\newcommand{\into}{\hookrightarrow}
\newcommand{\onto}{\twoheadrightarrow}

\newcommand{\xra}{\xrightarrow}

\newcommand{\wt}{\widetilde}

\newcommand{\om}{\Omega}
\newcommand{\si}{\Sigma}

\newcommand{\lsi}{\Omega\Sigma}
\newcommand{\lsii}{\Omega^2\Sigma^2}
\newcommand{\lsn}{\Omega^n\Sigma^n}

\newcommand{\mc}{\mathcal}

\newcommand{\mr}{\mathrm}
\newcommand{\ms}{\mathscr}

\newcommand{\F}{\mathbb{F}}

\newcommand{\R}{\mathbb{R}}
\newcommand{\Z}{\mathbb{Z}}

\newcommand{\id}{\mathrm{id}}

\newcommand{\Ker}{\mathrm{Ker}\,}

\newcommand{\Brun}{\mathrm{Brun}}

\begin{document}

\title[Operations on Spaces over Operads]{Operations on Spaces over Operads and Applications to Homotopy Groups}
\author{Wenbin ZHANG}
\date{30 Nov 2011}
\address{Department of Mathematics, National University of Singapore}
\email{smile.wenbin@gmail.com}

\begin{abstract}
  We establish certain smash operations on spaces over operads which are general analogues of the Samelson product on single loop spaces, and obtain a conceptual description of the structure of the homotopy groups of spaces $Y$ over a symmetric $K(\pi,1)$ operad: $\pi_*Y$ is a module over the free algebraic symmetric operad generated by operations on homotopy groups induced by these smash operations. In particular the homotopy groups of double loop spaces is a module over the free algebraic symmetric operad generated by the conjugacy classes of Brunnian braids modulo the conjugation action of pure braids.
\end{abstract}

\maketitle

\tableofcontents

\section{Introduction}
A fundamental problem in homotopy theory is to determine the homotopy groups $\pi_*X$ of a space $X$. Structures of homotopy groups would be important for the determination of $\pi_* X$ by comparing with (co)homology theories which have rich structures so that they may be determined by generators and certain structures. Thus structures are important to control (co)homology theories. For instance, the structures of $H_*\lsn X$ are the essential part in the determination of $H_*\lsn X$ \cite{Coh-Lad-May:1976:HILS}. The objective of this paper is to investigate the structures of homotopy groups. This paper is part the author's Ph.D. thesis.

Determination of the homotopy groups $\pi_*X$ of a space $X$ is equivalent to determination of the homotopy groups $\pi_* \om^n X$ of loop spaces $\om^n X$ ($n\geq 1$), as $\pi_{k+n}X= \pi_k \om^n X$. The latter has a great advantage that (iterated) loop spaces $\om^n X$ have rich structures which may be helpful to uncover the structures of $\pi_* \om^n X$ and thus the structures of $\pi_* X$. For instance, a single loop space $\om X$ admits a product
$$[-,-]: \om X \wedge \om X \to \om X$$
called the Samelson product, which induces a structure on $\pi_* \om X$ similar to a Lie algebra. As illustrated by this example, certain products on (iterated) loop spaces may induce certain structures on homotopy groups. In this paper, we are concerned about generalization of the Samelson product to iterated loop spaces, namely
\begin{quote}
  \textbf{Question:} What structures on $\om^n X$ ($n\geq 2$) can induce certain structures on $\pi_* \om^n X$ analogous to that the Samelson product on $\om X$ induces a Lie algebra structure on $\pi_* \om X$?
\end{quote}

To analyze this question, first let us recall that the little $n$-cubes operad $\ms{C}_n$ acts on $\om^n X$ for $n\geq 1$ \cite{Boa-Vog:1968:HEHS, May:1972:GILS} and its converse is also true.

\begin{thm}[May (1972) \cite{May:1972:GILS}, Boardman and Vogt (1973) \cite{Boa-Vog:1973:HIASTS}]
  If a path-connected space $Y$ admits an action of $\ms{C}_n$, then $Y$ is weakly homotopy equivalent to $\om^n X$ for some $X$.
\end{thm}
In other words, a path-connected space is of the weak homotopy type of an $n$-fold loop space iff it admits an action of $\ms{C}_n$ up to homotopy. Namely the action of $\ms{C}_n$ on $n$-fold loop spaces
$$\theta: \ms{C}_n(k) \times (\om^n X)^k \to \om^n X$$
characterizes $n$-fold loop spaces and thus should carry all the essential information of $n$-fold loop spaces. So good understanding of $\theta$ on certain aspects may be helpful to obtain certain information of $n$-fold loop spaces. For instance, in homology the behavior of $\theta$ is crucial to the homology of $n$-fold loop spaces, which had been well studied in 1970's \cite{Coh-Lad-May:1976:HILS}. Unfortunately on homotopy groups $\theta$, as well as $(\om \theta)_*$, is just the summation and thus does not carry useful information.

Besides homology, then what else important information of $n$-fold loop spaces can be extracted from $\theta$, especially on the space level? Notice that $\theta$ unites all operations on $n$-fold loop spaces as a whole and this unity is certainly of great advantages. In certain situations, however, it is necessary to break down $\theta$ into many finer operations. For example, in homology $\theta$ is automatically broken down into many homology operations. To extract information from $\theta$ on the space level, we propose the following idea which applies not only to $\ms{C}_n$ and $\om^n X$ but also to general topological operads $\ms{C}$ and $\ms{C}$-spaces:
\begin{quote}
  \textbf{Idea:} Break down the action $\theta$ of $\ms{C}$ on $\ms{C}$-spaces $Y$ into many finer operations by composing $\theta$ with elements in $[S^l, \ms{C}(k)]$, i.e. maps $S^l \to \ms{C}(k)$ to get various maps $S^l \times Y^k \to Y$, then assemble all these finer operations together to recover (partially) global structures of $Y$.
\end{quote}
Namely, for each $\alpha \in [S^l, \ms{C}(k)]$, let
$$\theta_{\alpha}: S^l\times Y^k \xra{\alpha \times \id^k} \ms{C}(k) \times Y^k \xra{\theta} Y,$$
then $\theta$ is broken down into a lot of product operations $\theta_{\alpha}$.

Note that $[S^l, \ms{C}]$ is naturally a $\Delta$-set with faces $d_i$. The key observation is that, if $d_i \alpha$ is trivial for all $i$, then $\theta_{\alpha}$ probably is homotopic to
$$\mu'_k: S^l \times Y^k \xra{\mr{proj.}} Y^k \xra{\mu_k} Y$$
restricted to the fat wedge where $\mu_k$ is the iterated product on $Y$; if so, then $\mu'_k- \theta_{\alpha}$ is null homotopic restricted to the fat wedge and thus factors through the smash product $S^l\wedge Y^{\wedge k}$, namely
$$\xymatrix{
  S^l \times Y^k \ar[d] \ar[r]^-{\mu'_k- \theta_{\alpha}}  & Y \\
  S^l\wedge Y^{\wedge k} \ar@{-->}[ur]_{\bar{\theta}_{\alpha}} }$$
where $\bar{\theta}_{\alpha}$ is the induced map, called a smash operation on $Y$, which can be thought of as a general analogue of the Samelson product. This in fact gives the Samelson product if $l=0$, $Y=\om X$ and $\alpha\in \pi_0 \ms{C}_1(2)= S_2$ is the transposition. Then each smash operation canonically induces a family of multilinear homomorphisms on homotopy groups
$$(\bar{\theta}_{\alpha})_*: \pi_l S^l \times \pi_{m_1}Y \times \cdots \times \pi_{m_k}Y\to \pi_{l+m_1+ \cdots+ m_k}Y,$$
sending $[f_i]\in \pi_{m_i} Y$ to the homotopy class of
$$S^{l+ m_1+ \cdots+ m_k}= S^l \wedge S^{m_1} \wedge \cdots \wedge S^{m_k} \xra{\id \wedge f_i \wedge \cdots \wedge f_k} S^l \wedge Y^{\wedge k} \xra{\bar{\theta}_{\alpha}} Y.$$
We actually need only consider
$$\tilde{\theta}_{\alpha}:= (\bar{\theta}_{\alpha})_* (\iota;-): \pi_{m_1}Y \times \cdots \times \pi_{m_k}Y \to \pi_{l+m_1+ \cdots+ m_k}Y,$$
where $\iota$ is the identity of $\pi_l S^l$.

We propose the following conjecture (Conjecture \ref{conj:smash_operation})

\begin{conj}
  Let $\ms{C}$ be a path-connected topological operad with a basepoint and $Y$ a $\ms{C}$-space, then for $\alpha\in [S^l, \ms{C}(k)]$ with all $d_i \alpha$ trivial, $\theta_{\alpha} \simeq \mu'_k$ restricted to the fat wedge of $S^l \times Y^k$ and thus $\mu'_k- \theta_{\alpha}$ induces a map $\bar{\theta}_{\alpha}: S^l\wedge Y^{\wedge k}\to Y$.
\end{conj}

We prove that this conjecture is true for the following two cases (Theorem \ref{thm:smash_operations-two} and Theorem \ref{thm:smash_operations-K(pi,1)}, respectively).

\begin{thm}
  The conjecture is true if 1) $k=2$, or 2) $Y$ is a path-connected topological $K(\pi,1)$ operad with the actions of symmetric groups free.
\end{thm}

The first case is proved by directly constructing a homotopy from $\theta_{\alpha}$ to $\mu'_k$. The proof for the second case given in this paper relies on a reconstruction of a $K(\pi,1)$ operad given in \cite{Zhang:2011:GOHT}. The approach is that this conjecture can be directly verified for the associated topological operad of a group operad, then it can be proved for a general $K(\pi,1)$ operad via the reconstruction of it from its fundamental groups operad.

For the case $\ms{C}_n$ and $\om^n_0 X$, The simplest smash operation (when $k=2$) is related to the Samelson product (they indeed coincide at least in homology) and its induced operation on homotopy groups is related to the Whitehead product. It is conjectured in this paper that they actually coincide.

By assembling all these induced operations on homotopy groups from smash operations on $\ms{C}$-spaces, we obtain the following conceptual description of the structure of the homotopy groups of $\ms{C}$-spaces (Theorem \ref{thm:structure_homotopy_groups-K(pi,1)}).

\begin{thm}
  If $\ms{C}$ is a topological $K(\pi,1)$ operad with the actions of symmetric groups free and $Y$ is a path-connected $\ms{C}$-space, then $\pi_*Y$ is a module over the free algebraic operad generated by all those $\alpha\in [S^l, \ms{C}]$ with $d_i \alpha$ trivial for all $i$. In particular, $\pi_* \om^2 X$ is a module over the free algebraic operad generated by the conjugacy classes of Brunnian braids modulo the conjugation action of pure braids.
\end{thm}

The identity map of $S^n$ ($n\geq 3$) particularly generates a family of elements in $\pi_*S^n$ under the action the conjugacy classes of Brunnian braids. It is also interesting to see (Remark \ref{rem:Brunnian-Lie}) that the conjugacy classes of Brunnian braids is related to $\mr{Lie}(n)$ due to Li and Wu \cite{LiWu:preprint:BGBBBHG}.

\textbf{Notations and conventions.}
For $\sigma,\tau\in S_n$, the product is $\sigma \cdot \tau := \tau \circ \sigma$, i.e. $(\sigma \cdot \tau)(i)= \tau (\sigma(i))$. Let $S_k$ acts on symmetric operads from left and on $X^k$ from right.

Given $k\geq 1$, $m_i\geq 0$ and $n_j\geq 0$, we shall often let $m= m_1+ \cdots+ m_k$, $n= n_1+ \cdots+ n_m$.

Two different label systems of $\Delta$-sets and simplicial sets are used here. One is the usual one starting from 0 and another one shifts 0 to 1, i.e., starting from 1. The latter is used for operads, like the symmetric groups operad, braid groups operad, etc.

For any symbol $a$, let $a^{(k)}$ denote the $k$-tuple $(a,\ldots,a)$.

For a normal subgroup $H$ of $G$, let $H/\mr{ca}(G)$ denote the set of conjugacy classes of $H$ modulo the conjugation action of $G$.

For two pointed spaces $X,Y$, let $\la X,Y \ra$ and $[X,Y]$ denote the sets of pointed and unpointed homotopy classes of maps $X\to Y$, respectively. Recall that \cite{Hatcher:2002:AT} if $X$ is a CW-complex and $Y$ is path-connected, then $\pi_1Y$ acts on $\la X,Y\ra$ and there is a natural bijection between $\la X,Y \ra/ \pi_1Y$ and $[X,Y]$; particularly $[S^1, Y]$ is the set of conjugacy classes of $\pi_1Y$.

Throughout this this paper, all topological spaces are assumed to be compactly generated Hausdorff spaces \cite{Steenrod:1967:CCTS}.

\textbf{Organization} of this paper is as follows. In section 2 we discuss some basic aspects of operads used in this paper. In section 3 we decompose the action of an operad $\ms{C}$ on a $\ms{C}$-space into many product operations and investigate their properties. In section 4 we investigate the existence of smash operations and the relation between the simplest smash operation and the Samelson product. In section 5 we investigate the induced operations on homotopy groups and their relation with the Whitehead product, and obtain a conceptual description of the structure of the homotopy groups of $\ms{C}$-spaces.

\section{Preliminaries}
In this section, we shall discuss some basic aspects of operads used in this paper.

For two points $x, x'$ in a space $X$, let $x\sim x'$ denote that $x,x'$ are in the same path-connected component of $X$.

\subsection{Product on $\ms{C}$-Spaces}
Let $\ms{C}$ be a nonsymmetric topological operad and $Y$ a $\ms{C}$-space with basepoint $*$. For $a\in \ms{C}(k)$ ($k\geq 1$),
$$\theta_a: Y^k\to a\times Y^k \into \ms{C}(k) \times Y^k \xra{\theta} Y$$
defines a product of $k$ variables. In addition, for $a\in \ms{C}(2)$, iteration of $\theta_a$ can be represented by $\gamma$ and $a$. For instance,
$$\theta_a \circ (\theta_a \times \id)= \gamma (a; \gamma (a;-),-)= \gamma (\gamma (a;a,1);-).$$

\begin{prop}
  If $a\sim b\in \ms{C}(k)$ ($k\geq 1$), then $\theta_a \simeq \theta_b$.
\end{prop}

\begin{proof}
  A path $f: I\to \ms{C}(k)$ with $f(0)=a$ and $f(1)=b$ gives
  $$F: I\times Y^k \xra{f\times \id^k} \ms{C}(k) \times Y^k \xra{\theta} Y$$
  which is a homotopy from $\theta_a$ to $\theta_b$.
\end{proof}

For $a\in \ms{C}(2)$, clearly $\theta_a(*,*)=*$, $\theta_a(*,y)= \theta(a;*,y)= \theta(d_1a;y)$, $\theta_a(y,*)= \theta(a;y,*)= \theta(d_2a;y)$.

\begin{prop}
  Let $\ms{C}$ be a nonsymmetric operad and $a\in \ms{C}(2)$.
  \begin{itemize}
    \item[1)] If $d_1a\sim d_2a\sim 1\in \ms{C}(1)$, then $\theta_a$ is homotopic to some $\mu: Y\times Y\to Y$ with $\mu(*,y)= \mu(y,*)=y$ via a basepoint preserving homotopy. If $d_1a= d_2a= 1$, then $\theta_a(*,y)=y$, $\theta_a(y,*)=y$.
    \item[2)] If $a\sim b\in \ms{C}(2)$, then $\theta_a\simeq \theta_b$ via a basepoint preserving homotopy.
    \item[3)] If $\gamma (a;1,a)\sim \gamma (a;a,1)\in \ms{C}(3)$, then $\theta_a$ is homotopy associative, thus $Y$ is a homotopy associative $H$-space with $\theta_a$. If $\gamma (a;1,a)= \gamma (a;a,1)$, then $\theta_a$ is associative.
  \end{itemize}
\end{prop}

\begin{proof}
  1) If $d_1a\sim d_2a\sim 1\in \ms{C}(1)$, then there are paths $f_1,f_2:I \to \ms{C}(1)$ such that $f_1(0)=d_1a$, $f_1(1)=1$, $f_2(0)=d_2a$, $f_2(1)=1$. Let $H_i: Y\times I\to Y$, $H_i(y,t)= \theta(f_i(t);y)$, $i=1,2$. Then $H_1$ is a homotopy from $\mu(*,-)$ to $\id$ with $H_1(*,t)= \theta(f_1(t);*)=*$, and $H_2$ is a homotopy from $\theta_a(-,*)$ to $\id$ with $H_2(*,t)= \theta(f_2(t);*)=*$. Moreover, $\theta_a$, $H_1$ and $H_2$ give a map
  $$(Y\times Y)\times 0 \cup (*\times Y \cup Y\times *) \times I \to Y.$$
  Using the homotopy extension property, $\theta_a$ is homotopic to some $\mu: Y\times Y\to Y$ with $\mu(*,y)= \mu(y,*)=y$ via a basepoint preserving homotopy. 2) and 3) can be proved similarly.
\end{proof}

\begin{prop}
  Let $\ms{C}$ be a symmetric operad, $a\in \ms{C}(2)$ and $\tau=(1,2)$ the transposition. If $a\sim \tau a \in \ms{C}(2)$, then $\theta_a$ is homotopy commutative. If $a= \tau a$, then $\theta_a$ is commutative. \qed
\end{prop}

\begin{cor}
  If there is $a\in \ms{C}(2)$ with $d_1a\sim d_2a\sim 1$, then $Y$ is an $H$-space with $\theta_a$; if moreover $\gamma (a;1,a)= \gamma (a;a,1)$ (resp. $\gamma (a;1,a)\sim \gamma (a;a,1)\in \ms{C}(3)$), then $Y$ is an associative (resp. homotopy associative) $H$-space with $\theta_a$; or if moreover $a=\tau a$ (resp. $a\sim \tau a \in \ms{C}(2)$), then $Y$ is a commutative (resp. homotopy commutative) $H$-space with $\theta_a$. \qed
\end{cor}

\subsection{Basepoint and Simplicial Structure of Operads}
We discuss some basic facts about basepoint and simplicial structure of operads. Some of the following definitions and facts are already given in \cite{Zhang:2011:GOHT}.

\begin{defn}
  A basepoint of a nonsymmetric operad $\ms{C}$ is a sequence of points $\{e_k\}_{k\geq 0}$ with $e_1=1\in \ms{C}(1)$, $e_k\in \ms{C}(k)$ such that $\gamma (e_k; e_{m_1}, \ldots, e_{m_k}) \sim e_m$. A strict basepoint is a basepoint such that $\gamma (e_k; e_{m_1}, \ldots, e_{m_k})= e_m$.
\end{defn}

For use in this paper, a basepoint of an operad is too week, but a strict basepoint is too strong, so that we need a kind of basepoint in between. Consider the little 1-cubes (intervals) operad $\ms{C}_1$ \cite{May:1972:GILS}. Let $c_0=*\in \ms{C}_1(0)$ and $c_k=([0,\frac{1}{k}], [\frac{1}{k}, \frac{2}{k}], \ldots, [\frac{k-1}{k},1]) \in \ms{C}_1(k)$. Let $\ms{C}_1(k)_0 \subseteq \ms{C}_1(k)$ be the path-connected component of $c_k$. Then $(\ms{C}_1)_0= \{\ms{C}_1(k)_0\}_{k\geq 0}$ is naturally a nonsymmetric suboperad of $\ms{C}_1$. Suppose $\ms{C}$ is a nonsymmetric topological operad admitting a morphism of nonsymmetric operads $\eta: (\ms{C}_1)_0\to \ms{C}$. For instance, if $\ms{C}$ has a strict basepoint, then $\eta: (\ms{C}_1)_0 \to *\into \ms{C}$; for the little $n$-cubes operad $\ms{C}_n$, there is a canonical inclusion $\eta: (\ms{C}_1)_0 \into \ms{C}_n$. Call $\{\eta(c_k)\}_{k\geq 0}$ a \textbf{good basepoint} of $\ms{C}$ and also say $\ms{C}$ is \textbf{well pointed}. For two well pointed topological operads $\ms{C}$ with $\eta: (\ms{C}_1)_0 \to \ms{C}$ and $\ms{C}'$ with $\eta': (\ms{C}_1)_0 \to \ms{C}'$, a morphism $\psi: \ms{C}\to \ms{C}'$ is called a morphism of well pointed topological operads if $\eta'= \psi\circ \eta$.

For a nonsymmetric operad $\ms{C}$, define
$$d_i: \ms{C}(k+1)\to \ms{C}(k), \quad d_ia= \gamma (a;1^{i-1},*,1^{k-i})$$
for $1\leq i\leq k+1$. If $\ms{C}$ has a basepoint $\{e_k\}_{k\geq 0}$, define
$$s_i: \ms{C}(k)\to \ms{C}(k+1), \quad s_ia= \gamma (a;1^{i-1},e_2,1^{k-i})$$
for $1\leq i\leq k$. By definition,
$$d_ie_{k+1}= \gamma (e_{k+1}; 1^{i-1}, *, 1^{k-i})\sim e_k, \quad s_ie_k= \gamma (e_k; 1^{i-1}, e_2, 1^{k-i})\sim e_{k+1}$$
for all $i$ and $k$.

\begin{prop}
  $\ms{C}$ has a basepoint iff there exists $a\in \ms{C}(2)$ such that $d_1a\sim d_2a\sim 1\in \ms{C}(1)$ and $\gamma (a;1,a)\sim \gamma (a;a,1)$. \qed
\end{prop}

\begin{prop}
  Let $\ms{C}$ be a nonsymmetric operad with a basepoint and $k\geq 1$. If $\ms{C}(k)$ is path-connected, then $\ms{C}(i)$ is also path-connected for each $i<k$.
\end{prop}

\begin{proof}
  For $a\in \ms{C}(k-1)$, $e_k\sim \gamma (e_2;1,a)$, thus $e_{k-1}\sim d_1e_k\sim d_1 \gamma (e_2;1,a)= \gamma (d_1e_2;a)\simeq \gamma (1;a)=a$. So the assertion holds.
\end{proof}

\begin{cor}
  Let $\ms{C}$ be a nonsymmetric operad and $k\geq 3$. $\ms{C}$ has a basepoint and $\ms{C}(k)$ is path-connected iff $\ms{C}(i)$ is also path-connected for each $i\leq k$. \qed
\end{cor}

\begin{prop}
  A nonsymmetric operad $\ms{C}$ is a $\Delta$-set. If $\ms{C}$ has a strict basepoint, then $\ms{C}$ is a simplicial set. If $\ms{C}$ has a basepoint, then $\ms{C}$ is a simplicial set up to homotopy, i.e. those simplicial identities hold up to homotopy. \qed
\end{prop}

\begin{prop}
  For a topological operad $\ms{C}$, $\pi_0 \ms{C}$ is a discrete operad and thus is a $\Delta$-set. If $\ms{C}$ has a basepoint, $\pi_0 \ms{C}$ also has a basepoint and thus is a simplicial set. \qed
\end{prop}

Let $A$ a connected CW complex with a vertex as the basepoint. The most interesting case is that $A$ is a Moore space, in particular $A= S^n$, $M(\Z/p^r, n)$.

\begin{prop}
  For a symmetric topological operad $\ms{C}$, $[A, \ms{C}]= \{[A, \ms{C}(k)]\}_{k\geq 0}$ with discrete topology is naturally a symmetric operad with the following data:
  \begin{itemize}
    \item[1)] For $[f]\in [A, \ms{C}(k)]$, $[g_i]\in [A, \ms{C}(m_i)]$, $\gamma ([f];[g_1],\ldots,[g_k])$ is defined as the homotopy class of the following composite
        $$A \xra{\Delta} A^{k+1} \xra{f \times g_1 \times \cdots \times g_k} \ms{C}(k) \times \ms{C}(m_1) \times \cdots \times \ms{C}(m_k) \xra{\gamma} \ms{C}(m).$$
    \item[2)] The unit is $[A\to 1 \into \ms{C}(1)]$.
    \item[3)] $S_k$ acts on $[A, \ms{C}(k)]$ by $\sigma \cdot [f]= [\sigma \cdot f]$ where $(\sigma \cdot f)(a)= \sigma (f(a))$ for $\sigma\in S_k$, $a\in \ms{C}(k)$.
  \end{itemize}
  If $\ms{C}$ has a basepoint $\{e_k\}_{k\geq 0}$, $[A, \ms{C}]$ also has a basepoint $\{[A \to e_k\into \ms{C}(k)]\}_{k\geq 0}$ and thus is a simplicial set. Moreover, if $A$ is a co-$H$-group and the action of $\pi_1 \ms{C}$ on $\la A,\ms{C} \ra$ is trivial, then $[A, \ms{C}]= \la A,\ms{C} \ra$ is a simplicial group. There are analogous results for nonsymmetric topological operads. \qed
\end{prop}

Some canonical operads admit additional structure which leads us to the following definition.

\begin{defn}
  A DDA-set is a sequence of sets $\{X_n\}_{n\geq 0}$ with deleting functions $d_i: X_{n+1} \to X_n$, doubling functions $s_i: X_{n+1}\to X_{n+2}$, and adding functions $d^i: X_n\to X_{n+1}$, $1\leq i\leq n+1$, $n\geq 0$, satisfying the following identities,
  $$d_i d_j= d_j d_{i+1}, \quad s_j s_i= s_{i+1} s_j, \quad d^j d^i= d^{i+1} d^j, \quad \textrm{for } j\leq i,$$
  $$d_j s_i =\left\{
  \begin{array}{cl}
    s_{i-1} d_j & j<i, \\
    \id & j=i, i+1, \\
    s_i d_{j-1} & j>i+1, \\
  \end{array}\right.
  d_j d^i =\left\{
  \begin{array}{cl}
    d^{i-1} d_j & j<i, \\
    \id & j=i, \\
    d^i d_{j-1} & j>i, \\
  \end{array}\right.
  s_j d^i =\left\{
  \begin{array}{cl}
    d^{i+1} s_j & j<i, \\
    d^i d^i & j=i, \\
    d^i s_{j-1} & j>i. \\
  \end{array}\right.$$
  A sequence of elements $\{e_n\}_{n\geq 0}$ with $e_n\in X_n$ is called a basepoint of a DDA-set $\{X_n\}_{n\geq 0}$ if $d_i e_n= e_{n-1}$, $s_i e_n= e_{n+1}$ and $d^i e_n= e_{n+1}$. A pointed DDA-set is a DDA-set with a basepoint. A morphism from a DDA-set $\{X_n\}_{n\geq 0}$ to another DDA-set $\{Y_n\}_{n\geq 0}$ is a sequence of functions $\{f_n: X_n\to Y_n\}_{n\geq 0}$ commuting with all $d_i$, $s_i$ and $d^i$. A pointed morphism from a pointed DDA-set $\{X_n\}_{n\geq 0}$ to another pointed DDA-set $\{Y_n\}_{n\geq 0}$ is a morphism preserving the basepoints. A DDA-group $\{G_n\}_{n\geq 0}$ is a DDA-set such that each $G_n$ is a group and all $d_i$, $s_i$ and $d^i$ are group homomorphisms.
\end{defn}

``DDA-set'' seems not a good name. We would like to use a better one if there is any. The sequence of sets $\{X^n\}_{n\geq 0}$ where $X$ is a pointed set, the sequence of symmetric groups $\{S_n\}_{n\geq 0}$, the sequence of braid groups $\{B_n\}_{n\geq 0}$ and the sequence of pure braid groups $\{P_n\}_{n\geq 0}$ are canonical examples of pointed DDA-set. $\{P_n\}_{n\geq 0}$ is moreover a DDA-group. A pointed DDA-set is obviously a contractible simplicial set. The three functions, deleting, doubling and adding functions are motivated by the example of braid groups.

We may also consider DDA-sets up to homotopy, i.e. those identities hold up to homotopy. The little $n$-cubes operad $\ms{C}_n$ ($n\geq 2$) is then not only a simplicial set up to homotopy, but also a DDA-set up to homotopy. Define $d^i: \ms{C}_n(k) \to \ms{C}_n(k+1)$ as follows. For $(c_1, \ldots, c_k)\in \ms{C}_n(k)$, suppose $c_j= c_{j1} \times \cdots \times c_{jn}$ where $c_{j1}, \ldots, c_{jn}: I\to I$ are affine functions. Define $d^i(c_1, \ldots, c_k)$ to be $(c'_1, \ldots, c'_{k+1})$, where $c'_j= (\frac{1}{2} c_{j1}) \times c_{j2} \times \cdots \times c_{jn}$ for $j<i$, $c'_j= (\frac{1}{2} c_{j-1,1}) \times c_{j-1,2} \times \cdots \times c_{j-1,n}$ for $j>i$, and $c'_i= (\frac{1}{2}+ \frac{1}{2} \id_I) \times \id_I^{n-1}$.

\begin{prop}
  $[S^1, \ms{C}_2]= \{P_n/ \mr{ca}(P_n)\}_{n\geq 0}$ is a DDA-set and $\pi_l\ms{C}_n= [S^l, \ms{C}_n]$ is an abelian DDA-group for $n\geq 3$. \qed
\end{prop}

\subsection{Structures on $\{[X\times Y^k, Y]\}_{k\geq 0}$} Recall that $\{\mr{Map}(Y^k,Y)\}_{k\geq 0}$ is an operad \cite{May:1972:GILS}. Similarly, we have
\begin{prop}
  $\{\mr{Map}(X\times Y^k, Y)\}_{k\geq 0}$ is a symmetric operad with the following data.
  \begin{itemize}
    \item[1)] For $f\in \mr{Map}(X\times Y^k, Y)$, $g_i\in \mr{Map}(X\times Y^{m_i}, Y)$, define $\gamma (f;g_1,\ldots,g_k)$ as the composite
        $$X \times Y^m \to X^{k+1} \times Y^m \to X \times (X\times Y^{m_1}) \times \cdots \times (X\times Y^{m_k}) \xra{\id_X \times g_1 \times \cdots \times g_k} X \times Y^k \xra{f} Y.$$
    \item[2)] The identity element is the projection $X\times Y\to Y$.
    \item[3)] $S_k$ acts on $\mr{Map}(X\times Y^k, Y)$ by $(\sigma\cdot f)(x;y_1,\ldots,y_k)= f(x; y_{\sigma^{-1}(1)}, \ldots, y_{\sigma^{-1}(k)})$ for $\sigma \in S_k$. \qed
  \end{itemize}
\end{prop}

\begin{cor}
  $\{[X\times Y^k, Y]\}_{k\geq 0}$ is a symmetric operad with the following data.
  \begin{itemize}
    \item[1)] For $[f]\in [X\times Y^k, Y]$, $[g_i]\in [X\times Y^{m_i}, Y]$,
        $$\gamma ([f];[g_1],\ldots,[g_k]):= [\gamma (f;g_1,\ldots,g_k)].$$
    \item[2)] The identity element is the homotopy class of the projection $X\times Y\to Y$.
    \item[3)] $S_k$ acts on $[X\times Y^k, Y]$ by $\sigma\cdot [f]= [\sigma \cdot f]$ for $\sigma \in S_k$. \qed
  \end{itemize}
\end{cor}

Let $Y$ be a homotopy associative $H$-space with product $\mu: Y^2 \to Y$. Let $\mu_k$ denote the iterated product $Y^k \xra{\mu_{k-1}\times \id} Y^2\xra{\mu} Y$, and $\mu'_k: X\times Y^k \xra{\mr{proj.}} Y^k \xra{\mu_k} Y$ (sometimes $[\mu'_k]$ is abbreviated to $\mu'_k$ for convenience).

\begin{prop}
  Suppose $Y$ is a homotopy associative $H$-space. Then $\{\mu'_k\}_{k\geq 0}$ is a basepoint of $\{\mr{Map}(X\times Y^k, Y)\}_{k\geq 0}$ and $\{[\mu'_k]\}_{k\geq 0}$ is a strict basepoint of $\{[X\times Y^k, Y]\}_{k\geq 0}$. \qed
\end{prop}

Let $Y$ be a homotopy associative and homotopy commutative $H$-space. Given $f\in \mr{Map}(X\times Y^k,Y)$, define $d_if\in \mr{Map} (X\times Y^{k-1},Y)$, $s_if, d^if\in \mr{Map} (X\times Y^{k+1},Y)$ as follows,
$$d_if= \gamma (f;\mu_1'^{i-1},*, \mu_1'^{k-i}), \quad s_if= \gamma (f;\mu_1'^{i-1}, \mu_2', \mu_1'^{k-i}),$$
$$d^if: X\times Y^{k+1}= X\times Y^{i-1} \times Y\times Y^{k+1-i} \to (X\times Y^{i-1} \times Y^{k+1-i}) \times Y \xra{f\times \id} Y\times Y \xra{\mu} Y,$$
$$(d^if) (x; y_1, \ldots, y_{k+1})= f(x; y_1, \ldots, \hat{y}_i, \ldots, y_{k+1})+ y_i.$$
Moreover, define $d_i[f]= [d_if]$, $s_i[f]=[s_if]$, $d^i[f]=[d^if]$. Note that generally $d^i([f]+[f']) \neq d^i[f]+ d^i[f']$, since
$$(d^i (f+f')) (x;y_1,\ldots,y_{k+1})= f(x; y_1, \ldots, \hat{y}_i, \ldots, y_{k+1})+ f'(x; y_1, \ldots, \hat{y}_i, \ldots, y_{k+1})+ y_i,$$
$$(d^if+ d^if') (x;y_1,\ldots,y_{k+1})= f(x; y_1, \ldots, \hat{y}_i, \ldots, y_{k+1})+ f'(x; y_1, \ldots, \hat{y}_i, \ldots, y_{k+1})+ 2y_i,$$
$$((d^if+ d^if')- d^i (f+f')) (x;y_1,\ldots,y_{k+1})= y_i.$$

\begin{lem}
  For $f, f'\in \mr{Map}(X\times Y^k, Y)$, $g_i\in \mr{Map}(X\times Y^{m_i}, Y)$,
  $$\gamma (f+f'; g_1, \ldots, g_k)= \gamma (f;g_1,\ldots,g_k)+ \gamma (f';g_1,\ldots,g_k).$$
  Thus
  $$\gamma ([f]+[f']; [g_1], \ldots, [g_k])= \gamma ([f];[g_1],\ldots,[g_k])+ \gamma ([f'];[g_1],\ldots,[g_k]).$$
\end{lem}

\begin{proof}
  The identity follows from the following commutative diagram
  \begin{diagram}
    X\times Y^m  &\rTo^{\Delta \times \id^m}  &X^{k+1} \times Y^m \\
    \dTo^{\Delta}  &&\dTo \\
    (X\times Y^m)^2  &&X \times (X\times Y^{m_1}) \times \cdots\\
    \dTo<{(\Delta \times \id^m)^2}  &&\dTo>{\id \times g_1 \times \cdots \times g_k} \\
    (X^{k+1} \times Y^m)^2  &&X\times Y^k \\
    \dTo  &&\dTo>{\Delta} \\
    (X\times (X\times Y^{m_1}) \times \cdots)^2  &\rTo<{\id \times g_1\times \cdots \times g_k}  &(X\times Y^k)^2  &\rTo^{f\times f'}  &Y^2  &\rTo^{\mu}  &Y.
  \end{diagram}
\end{proof}

\begin{rem}
  In general, $\gamma$ is not multilinear. In fact, generally,
  $$\gamma (f; g_1+g'_1, g_2, \ldots, g_k) \not \simeq \gamma (f;g_1,\ldots,g_k)+ \gamma (f;g'_1,\ldots,g_k),$$
  For example,
  $$\gamma (\mu'_k; g_1+g'_1, g_2, \ldots, g_k) \simeq g_1+g'_1 +g_2+\cdots+g_k,$$
  $$\gamma (\mu'_k;g_1,\ldots,g_k)+ \gamma (\mu'_k;g'_1,\ldots,g'_k) \simeq g_1+\cdots+g_k + g'_1+\cdots+g'_k.$$
  Moreover, generally
  $$\gamma (f; g_1+g'_1, \ldots, g_k+g'_k) \not \simeq \gamma (f;g_1,\ldots,g_k) + \gamma (f;g'_1,\ldots,g'_k).$$
\end{rem}

\begin{prop}
  For a homotopy associative and homotopy commutative $H$-space $Y$, $\{[X\times Y^k, Y]\}_{k\geq 0}$ is a simplicial abelian group and a DDA-set. \qed
\end{prop}

\section{Product Operations on $\ms{C}$-Spaces}
In this section, the action of a (nonsymmetric or symmetric) operad $\ms{C}$ on a $\ms{C}$-space is broken down into many small pieces, product operations.

Let $\ms{C}= \{\ms{C}(k)\}_{k\geq 0}$ be a (nonsymmetric or symmetric) operad and $Y$ be a $\ms{C}$-space with action $\theta: \ms{C}(k) \times Y^k \to Y$, $k\geq 0$. For $a\in \ms{C}(k)$, define
$$\theta_a: Y^k \to a\times Y^k \into \ms{C}(k)\times Y^k \xra{\theta} Y.$$
If $a\sim b\in \ms{C}(k)$ then $\theta_a\simeq \theta_b$. Thus there are obvious functions
$$\Theta_{\ms{C}}^{0,k}: \pi_0 \ms{C}(k)\to [Y^k,Y], \quad \alpha=[a] \mapsto \theta_{\alpha}= [\theta_a].$$
For $f: S^l \to \ms{C}(k)$ ($l>0$), define
$$\theta_f: S^l \times Y^k \xra{f \times \id^k} \ms{C}(k) \times Y^k \xra{\theta} Y.$$
Clearly $\theta_f\simeq \theta_g$ if $f\simeq g$. Thus we have functions
$$\Theta_{\ms{C}}^{l,k}: [S^l, \ms{C}(k)] \to [S^l \times Y^k, Y], \quad \alpha= [f] \mapsto \theta_{\alpha}= [\theta_f].$$
In particular,
$$\Theta_2^k= \Theta_{\ms{C}_2}^{1,k}: P_k/\mr{ca}(P_k)= [S^1, \ms{C}_2(k)] \to [S^1 \times (\om^2 X)^k, \om^2 X], \quad \alpha \mapsto \theta_{\alpha},$$
and for $n\geq 3$,
$$\Theta_n^{l,k}= \Theta_{\ms{C}_n}^{l,k}: \pi_l \ms{C}_n(k)= [S^l, \ms{C}_n(k)] \to [S^l \times (\om^n X)^k, \om^n X], \quad \alpha \mapsto \theta_{\alpha}.$$
For $\alpha\in \pi_0\ms{C}(k)$ or $\alpha\in [S^l, \ms{C}(k)]$, $\theta_{\alpha}$ is obviously natural with respect to $\ms{C}$-maps. Call these $\theta_{\alpha}$ product operations on $Y$.

\begin{rem}
  In the construction of product operations, $S^l$ can be replaced by a general path-connected space, in particular by a Moore space.
\end{rem}

\subsection{Behavior of Product Operations in Homology}
For $\ms{C}= \ms{C}_n$ ($n\geq 2$), the behavior of these product operations in homology can be calculated following Cohen's calculation \cite{Coh-Lad-May:1976:HILS} of the homology of $\ms{C}_n$-spaces. From this calculation, we shall see that many product operations are nontrivial.

For convenience, the following two composites (with integral coefficients)
\begin{align*}
  H_*\ms{C}(k) \otimes (H_*Y)^{\otimes k} \to H_*(\ms{C}(k) \times Y^k) \xra{\theta_*} H_*Y, \\
  H_lS^l \otimes (H_*Y)^{\otimes k} \to H_*(S^l \times Y^k) \xra{(\theta_{\alpha})_*} H_*Y
\end{align*}
are still denoted by $\theta_*$ and $(\theta_{\alpha})_*$, respectively. Clearly
$$(\theta_{\alpha})_* (\iota \otimes y_1 \otimes \cdots \otimes y_k)= \theta_* (\bar{\alpha} \otimes y_1 \otimes \cdots \otimes y_k),$$
where $\iota$ is the fundamental class of $H_lS^l$ and $\bar{\alpha}= \alpha_*(\iota)$ is the homology class of $\alpha$.

\begin{prop}
  Let $a_{ij}$, $1\leq i< j\leq k$, denote the standard generators of $H_{n-1} (\ms{C}_n(k))$ ($n\geq 2$). Then
  $$\theta_* (a_{ij} \otimes x_1 \otimes \cdots \otimes x_k)= \lambda_{n-1} (x_i, x_j) x_1\cdots \hat{x}_i \cdots \hat{x}_j \cdots x_k$$
  up to sign, where $x_1, \ldots, x_k\in H_*\om^n X$ and $\lambda_{n-1} (x,y)= (-1)^{(n-1)|x|+1} \theta_* (\iota \otimes x\otimes y)$ is the Browder operation \cite{Coh-Lad-May:1976:HILS}.
\end{prop}

\begin{proof}
  $a_{ij}$ can be represented by a map $S^{n-1}\to \ms{C}_n(k)$ given by an inclusion $S^{n-1}\to \R^n$ illustrated by the following figure.
  \begin{center}
    \includegraphics[width=0.4\textwidth]{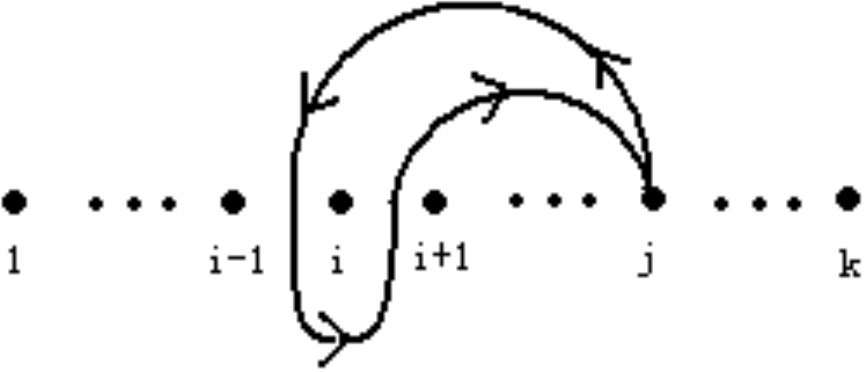}
  \end{center}
  (This map is homotopic to two other maps illustrated by the following two figures respectively.
  \begin{center}
    \includegraphics[width=0.4\textwidth]{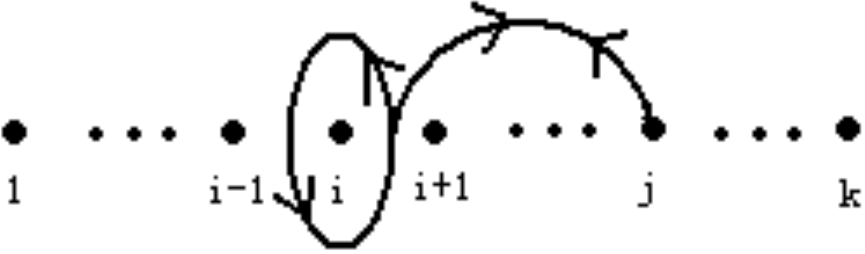}
    \hspace{0.1\textwidth}
    \includegraphics[width=0.4\textwidth]{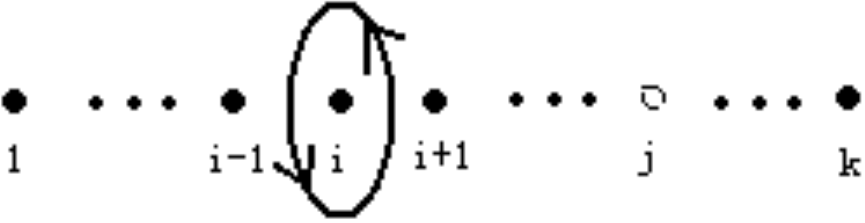}
  \end{center}
  The latter is given in Cohen's calculation.) Note that
  $$\gamma_*: H_*\ms{C}_n(2) \otimes H_*\ms{C}_n(2) \otimes H_*\ms{C}_n(k-2) \to H_*\ms{C}_n(k), \quad \gamma_*(e_0 \otimes a_{12} \otimes e_0)= a_{12}.$$
  Then
  \begin{align*}
    \theta_*(a_{12} \otimes x_1 \otimes \cdots \otimes x_k) & = \theta_*(\gamma_*(e_0 \otimes a_{12} \otimes e_0) \otimes x_1 \otimes \cdots \otimes x_k) \\
    & = \theta_* (e_0 \otimes \theta_*(a_{12} \otimes x_1 \otimes x_2)\otimes \theta_*(e_0\otimes x_3 \otimes \cdots \otimes x_k)) \\
    & = \theta_*(a_{12} \otimes x_1 \otimes x_2) (x_3 \cdots x_k) \\
    & = (-1)^{(n-1)|x|+1} \lambda_{n-1} (x_1, x_2) x_3\cdots x_k.
  \end{align*}
  The others follow from the identity for $a_{12}$ and the equivariance of $\theta$, c.f. sections 7, 12 and 13 of Cohen's calculation of the homology of $\ms{C}_n$-spaces \cite{Coh-Lad-May:1976:HILS}.
\end{proof}

\begin{cor}
  Let $A_{ij}$, $1\leq i< j\leq k$, denote the standard generators of $\pi_1 \ms{C}_2(k)= P_k$. Then
  $$(\theta_{A_{ij}})_* (\iota \otimes x_1 \otimes \cdots \otimes x_k)= \lambda_1 (x_i, x_j)  x_1\cdots \hat{x}_i \cdots \hat{x}_j \cdots x_k$$
  up to sign, where $x_1, \ldots, x_k\in H_*\om^2 X$. \qed
\end{cor}

The sign can be made explicitly. It is omitted here just for simplicity.

Consequently, many of these operations are essential; namely, not null-homotopic, nor homotopic to the summation
$$\mu'_k: S^l \times (\om^n X)^k \xra{\mr{proj.}} (\om^n X)^k \xra{\mu_k} \om^n X.$$

\subsection{Structures Preserved by Product Operations}
These product operations preserve certain combinatorial structures.

\begin{prop}
  For $l\geq 0$, $\Theta_{\ms{C}}^{l,*}= \{\Theta_{\ms{C}}^{l,k}\}_{k\geq 0}$ is a morphism of nonsymmetric (resp. symmetric) operads and preserves basepoints if $\ms{C}$ is nonsymmetric (resp. symmetric).
\end{prop}

$\Theta_{\ms{C}}^{*,*}$ may be thought of as a measure of the complexity of the $\ms{C}$-structure on $Y$, in particular as a measure of the complexity of the $n$-fold loop space structure on $\om^n X$ if $\ms{C}= \ms{C}_n$ for $n\geq 2$. In addition, product operations provide examples of another kind of actions of operads, i.e. the operad $[S^l, \ms{C}]$ acts on $Y$ via
$$\Theta_{\ms{C}}^{l,*}\to \{[S^l \times Y^k,Y]\}_{k\geq 0}.$$

\begin{proof}
  The assertion is obvious for $l=0$. Next let $l\geq 1$. For convenience, abbreviate $\Theta_{\ms{C}}^{l,k}$ to $\Theta$. First check that $\Theta$ commutes with $\gamma$. For $[f]\in [S^l, \ms{C}(k)]$ and $[g_i] \in [S^l, \ms{C}(m_i)]$,
  \begin{align*}
    \Theta (\gamma ([f];[g_1],\ldots,[g_k])) & = \theta_{\gamma ([f];[g_1],\ldots,[g_k])}, \\
    \gamma (\Theta[f]; \Theta[g_1], \ldots, \Theta[g_k]) & = \gamma (\theta_{[f]}; \theta_{[g_1]}, \ldots, \theta_{[g_k]}).
  \end{align*}
  To prove
  $$\Theta (\gamma ([f];[g_1],\ldots,[g_k]))= \gamma (\Theta[f]; \Theta[g_1], \ldots, \Theta[g_k]),$$
  we need only prove
  $$\gamma (\theta_{[f]}; \theta_{[g_1]}, \ldots, \theta_{[g_k]})= \theta_{\gamma ([f];[g_1],\ldots,[g_k])}.$$
  The latter identity follows from the following homotopy commutative diagram
  \begin{diagram}
    &&S^l \times Y^m\\
    &\ldTo<{\id \times (g_1 \times \id^{m_1}) \times \cdots}  &\dTo<{f \times (g_1 \times \id^{m_1}) \times \cdots}  &\rdTo>{\gamma(f; g_1, \ldots, g_k) \times \id^m}\\
    S^l \times (\ms{C}(m_1) \times Y^{m_1}) \times \cdots &&\ms{C}(k) \times (\ms{C}(m_1) \times Y^{m_1}) \times \cdots  &\rTo^{\gamma \times \id^m}  &\ms{C}(m) \times Y^m \\
    \dTo<{\id \times \theta^k}  &&\dTo<{\id \times \theta^k}  &&\dTo>{\theta}\\
    S^l\times Y^k  &\rTo^{f \times \id^k} &\ms{C}(k) \times Y^k  &\rTo^{\theta}  &Y,
  \end{diagram}
  where the left part and the right-lower square are strictly commutative and the right-upper triangle is homotopy commutative.

  It is evident that $\Theta$ preserves the identity element.

  Finally if $\ms{C}$ is symmetric, the equivariance of $\Theta$ follows from the following commutative diagram
  \begin{diagram}
    S^l \times Y^k  &\rTo^{f\times \id^k}  &\ms{C}(k) \times Y^k  &\rTo^{\sigma \times \id^k}  &\ms{C}(k) \times Y^k \\
    \dTo<{\id\times \sigma}  &\rdTo^{f\times \sigma}  &\dTo>{\id\times \sigma}  &&\dTo>{\theta} \\
    S^l\times Y^k  &\rTo^{f\times \id^k}  &\ms{C}(k) \times Y^k &\rTo^{\theta}  &Y.
  \end{diagram}
\end{proof}

\begin{prop}
  For $l\geq 1$, $\Theta_{\ms{C}}^{l,*}$ is natural with respect to $\ms{C}$-maps in the sense that the following diagram is homotopy commutative
  \begin{diagram}
    S^l \times Y_1^k  &\rTo^{\theta_a}  &Y_1 \\
    \dTo<{\id\times f^k}  &&\dTo>f \\
    S^l \times Y_2^k  &\rTo^{\theta_a}  &Y_2
  \end{diagram}
  where $a: S^l \to \ms{C}(k)$ and $f: Y_1\to Y_2$ is a $\ms{C}$-map. There is similar naturality for $\Theta^{0,*}_{\ms{C}}$.
\end{prop}

\begin{proof}
  The above diagram is actually commutative due to the commutativity of the following one
  \begin{diagram}
    S^l \times Y_1^k  &\rTo^{a\times \id^k}  &\ms{C}(k) \times Y_1^k  &\rTo^{\theta}  &Y_1 \\
    \dTo<{\id\times f^k}  &&\dTo<{\id\times f^k}  &&\dTo>f \\
    S^l \times Y_2^k  &\rTo^{a\times \id^k}  &\ms{C}(k) \times Y_2^k  &\rTo^{\theta}  &Y_2.
  \end{diagram}
\end{proof}

\begin{prop}
  For $l\geq 0$, $\Theta_{\ms{C}}^{l,*}$ is a morphism of simplicial sets. For $n\geq 2$, $\Theta_n^{l,*}$ is moreover a morphism of DDA-sets.
\end{prop}

\begin{proof}
  It remains to check that
  $$d^i \circ \Theta_n^{l,k}= \Theta_n^{l,k+1} \circ d^i$$
  which can be proved first for $d^{k+1}$, then for the rest by converting $d^i$ to $d^{k+1}$ using the equivariance of $\theta$.
\end{proof}

It should be noted that $\Theta_n^{l,k}$ is not a group homomorphism.

\section{Smash Operations on $\ms{C}$-Spaces}
Some product operations $Y^k\to Y$ (resp. $S^l \times Y^k\to Y$) may induce maps $Y^{\wedge k}\to Y$ (resp. $S^l \wedge Y^{\wedge k}\to Y$) which are natural with respect to $\ms{C}$-maps. Call these induced maps smash operations on $Y$. These smash operations are interesting because they are general analogues of the Samelson product and provide many operations on homotopy groups.

\subsection{Preparation}
Let $\ms{C}$ be a (nonsymmetric or symmetric) operad with a basepoint $\{e_k\}_{k\geq 0}$, $Y$ a $\ms{C}$-space and $Y_0$ the path-connected component of the basepoint of $Y$. Let
\begin{align*}
  \mu_k &= \theta_{e_k}: Y^k\to e_k\times Y^k \into \ms{C}(k) \times Y^k \xra{\theta} Y, \\
  \mu'_k &= S^l\times Y^k \xra{\mr{proj.}} Y^k \xra{\mu_k} Y.
\end{align*}
Note that $(k-1)$-fold iteration of $\mu_2$ is homotopic to $\mu_k$.

\begin{lem}
  For $l\geq 1$, $\theta_f \circ (\id_{S^l} \times d^i)= \theta_{d_i f}$,
\begin{diagram}
  S^l \times Y^{k-1}  &\rTo^{\id_{S^l} \times d^i}  &S^l \times Y^k \\
  &\rdTo_{\theta_{d_i f}}  &\dTo>{\theta_f} \\
  &&Y,
\end{diagram}
where $f:S^l\to \ms{C}(k)$ and $d^i: Y^{k-1}= Y^{i-1} \times * \times Y^{k-i} \into Y^k$. In particular $\theta_f \circ (\id_{S^l} \times d^i)= \theta_{d_if} \simeq \mu'_{k-1}$ if $d_if$ is nullhomotopic.
\end{lem}

\begin{proof}
  By the definition of a $\ms{C}$-action, the following diagram is commutative
  \begin{diagram}
    \ms{C}(k) \times Y^{i-1} \times * \times Y^{k-i}  &\rInto  &\ms{C}(k) \times Y^k \\
    \dTo<{d_i}  &&\dTo_{\theta} \\
    \ms{C}(k-1) \times Y^{k-1} &\rTo^{\theta}  &Y,
  \end{diagram}
  and $\theta(c;*^k)=*$ for any $c\in \ms{C}(k)$. Then the following diagram
  \begin{diagram}
    &&S^l \times Y^{i-1}\times *\times Y^{k-1}  &\rInto  &S^l\times Y^k \\
    &\ruEq  &\dTo  &&\dTo>{f\times \id} \\
    S^l\times Y^{k-1}  &\rTo^{f\times \id}  &\ms{C}(k) \times Y^{i-1}\times *\times Y^{k-1}  &\rInto  &\ms{C}(k) \times Y^k \\
    &\rdTo_{d_if \times \id}  &\dTo>{d_i}  &&\dTo>{\theta} \\
    &&\ms{C}(k-1) \times Y^{k-1}  &\rTo^{\theta}  &Y
  \end{diagram}
  is commutative, thus the identity holds.
\end{proof}

We are interested in those $\alpha\in \pi_0 \ms{C}$ or $\alpha\in [S^l, \ms{C}(k)]$ ($l\geq 1$) with $d_i\alpha$ trivial for all $i$. Let $\mc{Z}_1 \pi_0 \ms{C}=1$ and for $k>1$,
$$\mc{Z}_k \pi_0 \ms{C}= \bigcap_{i=1}^k \Ker (d_i: \pi_0\ms{C}(k) \to \pi_0\ms{C}(k-1)).$$
If $\ms{C}$ is discrete, also denote $\Brun_k \ms{C}= \mc{Z}_k \pi_0 \ms{C}$. For instance, $\Brun_k \ms{P}= \Brun_k$ is the usual Brunnian braid group, where $\ms{P}= \{P_k\}_{k\geq 0}$ is the sequence of pure braid groups. For $l\geq 1$, let $\mc{Z}_1 [S^l, \ms{C}]=1$ and for $k>1$,
$$\mc{Z}_k [S^l, \ms{C}]= \bigcap_{i=1}^k \Ker (d_i: [S^l, \ms{C}(k)] \to [S^l, \ms{C}(k-1)]).$$
For example, $\mc{Z}_2 [S^1, \ms{C}_2]= \mr{Brun}_2/ \mr{ca}(P_2)=P_2$, $\mc{Z}_k [S^1, \ms{C}_2]= \mr{Brun}_k/ \mr{ca}(P_k)$ and for $n\geq 3$, $\mc{Z}_2 [S^l, \ms{C}_n]= \pi_l \ms{C}_n(2)= \pi_l S^{n-1}$, $\mc{Z}_k [S^l, \ms{C}_n]= \mc{Z}_k \pi_l \ms{C}_n$ is a subgroup of \cite{May:1972:GILS}
$$\pi_l\ms{C}_n(k)\approx \bigoplus_{i=1}^{k-1} \pi_l (\bigvee^i S^{n-1})$$
which may be thought of as the higher dimensional analogue of Brunnian braids. (Note \cite{May:1972:GILS} $\ms{C}_n(k)\simeq F(\R^n,k)$, $\pi_l \ms{C}_n(k)= \pi_l F(\R^n,k)$, $\pi_1 F(\R^2,k)= P_k$.)

For the product $X_1\times \cdots \times X_k$ of pointed spaces, define its fat wedge as
$$\mr{FW}(X_1\times \cdots \times X_k)= \{(x_1,\ldots,x_k) \in X_1\times \cdots \times X_k \mid \textrm{at least one } x_i=*\}.$$
Note that if $\alpha=[f] \in \mc{Z}_k [S^l, \ms{C}]$, then $\theta_{f} \circ (\id_{S^l} \times d^i)= \theta_{d_i f}\simeq \mu'_{k-1}$ for all $i$. (For convenience, we shall sometimes also use $\theta_{\alpha}$ to denote $\theta_f$ by abuse of notation.) One may wonder if all these homotopies can be glued together to give a homotopy on the fat wedge. Namely, is $\theta_{\alpha}$ homotopic to $\mu'_k$ restricted to the fat wedge,
$$\theta_{\alpha} \simeq \mu'_k: \mr{FW}(S^l\times Y^k)= (*\times Y^k) \cup (S^l\times \bigcup_{i=1}^{k} Y^{i-1}\times *\times Y^{k-i}) \to Y?$$
Noting that $S^l\wedge Y^{\wedge k}= S^l \times Y^k/ \mr{FW} (S^l\times Y^k)$, if this is true, then $\mu'_k- \theta_{\alpha}$ factors through $S^l\wedge Y^{\wedge k}$,
$$\xymatrix{
  S^l \times Y^k \ar[d] \ar[r]^-{\mu'_k- \theta_{\alpha}}  & Y \\
  S^l\wedge Y^{\wedge k} \ar@{-->}[ur]_{\bar{\theta}_{\alpha}} }$$
where $\bar{\theta}_{\alpha}$ is the induced map, called a smash operation on $Y$. Hence to get smash operations, the key is to have that $\theta_{\alpha}$ is homotopic to $\mu'_k$ restricted to the fat wedge.

\begin{example}
  For the little 1-cubes operad $\ms{C}_1$, $\pi_0 \ms{C}_1 \cong \ms{S}= \{S_k\}_{k\geq 0}$ the symmetric groups operad, $\mc{Z}_2 \pi_0 \ms{C}_1= S_2= \{1,\tau\}$ and $\mc{Z}_k \pi_0 \ms{C}_1=1$ for $k\neq 2$ where $\tau=(12)$ is the transposition. Let $\mu: \om X\times \om X\to \om X$ be the loop product. Then
  $$\theta_{\tau}: \om X\times \om X\to \om X, \quad (f,g) \mapsto \mu (g,f)= g\cdot f.$$
  Evidently, $\theta_{\tau} \simeq \mu: \om X\vee \om X \to \om X$. Thus $\mu- \theta_{\tau}: \om X \times \om X\to \om X$ factors through $\om X\wedge \om X$,
  $$\xymatrix{
  \om X \times \om X \ar[d] \ar[r]^-{\mu- \theta_{\tau}}  & \om X \\
  \om X \wedge \om X \ar@{-->}[ur]_{\bar{\theta}_{\tau}} }$$
  Clearly $\mu- \theta_{\tau}$ is exactly the commutator product $[f,g]= fgf^{-1} g^{-1}$. Thus the induced map $\bar{\theta}_{\tau}$ is exactly the Samelson product $[-,-]: \om X\wedge \om X\to \om X$.
\end{example}

Concerning the existence of smash operations, we propose the following conjecture.

\begin{conj}\label{conj:smash_operation}
  Let $\ms{C}$ be an operad with a basepoint and $k\geq 2$. If $\alpha\in \mc{Z}_k [S^l, \ms{C}]$ ($l\geq 1$) such that $\alpha(*)\sim e_k$, then
  $$\theta_{\alpha} \simeq \mu'_k: \mr{FW}(S^l\times Y^k) \to Y.$$
\end{conj}

From the above discussion, if this conjecture is true for an operad $\ms{C}$, then each $\alpha\in \mc{Z}_k [S^l, \ms{C}]$ gives a smash operation $\bar{\theta}_{\alpha}: S^l \wedge Y^{\wedge k} \to Y$.

\subsection{The Case of Two Factors}
In this subsection we prove that Conjecture \ref{conj:smash_operation} is true in the case of two factors for general operads. For $a\in \ms{C}(k)$, let $\theta'_a: S^l\times Y^k \xra{\mr{proj.}} Y^k \xra{\theta_a} Y$, and $\epsilon_1: S^l\to 1\into \ms{C}(1)$, $\epsilon_a: S^l\to a\into \ms{C}(k)$ denote the constant maps.

\begin{thm}\label{thm:smash_operations-two}
  For an operad $\ms{C}$, if $e\in \ms{C}(2)$ with $d_1e\sim d_2e\sim 1\in \ms{C}(1)$ and $\alpha \in \mc{Z}_2 [S^l, \ms{C}]$ ($l\geq 1$) with $\alpha(*)\sim e$, then
  $$\theta_{\alpha} \simeq \theta'_e: \mr{FW}(S^l\times Y^2)= (* \times Y^2) \cup (S^l \times * \times Y) \cup (S^l \times Y \times *) \to Y.$$
\end{thm}

This result should also be valid for homotopy commutative $H$-spaces.

\begin{proof}
  Suppose $\alpha=[f]\in \mc{Z}_2 [S^l, \ms{C}]$, namely $f: S^l\to \ms{C}(2)$ with $d_1f\simeq \epsilon_1$ and $d_2f\simeq \epsilon_1$. We may assume $f(*)=e$ (if not, choose $g\simeq f$ with $g(*)=e$, then $\theta_f\simeq \theta_{g}$). Then $\theta_f= \theta'_e$ restricted to $*\times Y^2\subset S^l\times Y^2$.

  First construct a homotopy from $\theta_f$ to $\theta'_e$ restricted to $S^l\times (Y\vee Y)$. Note
  \begin{align*}
    \theta_f (a;*,y)= \theta (f(a);*,y)= \theta (d_1f(a);y), \quad \theta'_e (a;*,y)= \theta (e;*,y)= \theta (d_1e;y), \\
    \theta_f (a;y,*)= \theta (f(a);y,*)= \theta (d_2f(a);y), \quad \theta'_e (a;y,*)= \theta (e;y,*)= \theta (d_2e;y).
  \end{align*}
  Since $d_1f\simeq \epsilon_1\simeq \epsilon_{d_1e}$ and $d_2f\simeq \epsilon_1\simeq \epsilon_{d_2e}$, there are $F_1,F_2: I\times S^l\to \ms{C}(1)$ such that $F_1(0;a)= d_1f(a)$, $F_1(1;a)=d_1e$, $F_2(0;a)= d_2f(a)$, $F_2(1;a)=d_2e$. Define
  \begin{align*}
    F'_1: I\times S^l\times *\times Y \to \ms{C}(1)\times Y\to Y, \quad F'_1(t;a; *,y)= \theta (F_1(t;a);y), \\
    F'_2: I\times S^l\times Y\times * \to \ms{C}(1)\times Y\to Y, \quad F'_2(t;a; y,*)= \theta (F_2(t;a);y).
  \end{align*}
  Then $F'_1(0;a;y,*)= \theta (d_1f(a);y)= \theta_f (a;*,y)$, $F'_1(1;a;*,y)= \theta (d_1e;y)= \theta'_e (a;*,y)$, $F'_1(t;a;*,*)= \theta (F_1(t;a); *)=*$; $F'_2(0;a;*,y)= \theta (d_2f(a);y)= \theta_f (a;y,*)$, $F'_2(1,a;y,*)= \theta (d_2e;y)= \theta'_e (a;y,*)$, $F'_2(t;a;*,*)= \theta (F(t;a); *)=*$. $F'_1,F'_2$ coincide on
  $$(I\times S^l\times *\times Y) \cap (I\times S^l\times Y\times *)= I\times S^l\times *\times *,$$
  thus they can be glued together to give a homotopy
  $$H: I\times S^l\times (*\times Y\cup Y\times *) \to Y$$
  such that $H(0;a;y_1,y_2)= \theta_f(a;y_1,y_2)$, $H(1;a;y_1,y_2)= \theta'_e (a;y_1,y_2)$.

  Note that $H(t;*;y_1,y_2)$ may not be $\theta'_e (*;y_1,y_2)$ since we may not have $F(t;*)=d_1e$, $G(t;*)=d_2e$ (the loops $H(-;*;y_1,y_2)$ even may not be nullhomotopic). Next modify $H$ to another homotopy $H'$ from $\theta_f$ to $\theta'_e$ with $H'(t;*;y_1,y_2)= \theta'_e (*;y_1,y_2)$ for $t\in I$ and $(*;y_1,y_2)\in (*\times Y^2) \cap (S^l\times (Y\vee Y))= *\times (Y\vee Y)$, then a homotopy from $\theta_f$ to $\theta'_e$ restricted to the fat wedge is obtained. $H'$ can be obtained in the following way. Let
  $$H_1: I\times S^l\times (Y\vee Y) \to Y, \quad (t;a;y_1,y_2) \mapsto H(1-t;*;y_1,y_2).$$
  Gluing $H$ and $H_1$ together, let
  $$H_2: I\times S^l \times (Y\vee Y) \to Y, \quad
  H_2(t;a;y_1,y_2)= \left\{%
  \begin{array}{ll}
    H(2t;a;y_1,y_2) & 0\leq t\leq \frac{1}{2}, \\
    H(2-2t;*;y_1,y_2) & \frac{1}{2}\leq t\leq 1. \\
  \end{array}\right.$$
  (Now the loops $H_2(-;*;y_1,y_2)$ are uniformly nullhomotopic.) Define a map $I\times I\times * \times (Y\vee Y)\to Y$
  (which is a uniform nullhomotopy of the loops $H_2(-;*;y_1,y_2)$),
  $$(s,t;*;y_1,y_2)\mapsto \left\{%
  \begin{array}{ll}
    H(2t;*;y_1,y_2) & 0\leq t\leq \frac{1-s}{2}, \\
    H(1-s;*;y_1,y_2) & \frac{1-s}{2}\leq t\leq \frac{1+s}{2}, \\
    H(2-2t;*;y_1,y_2) & \frac{1+s}{2}\leq t\leq 1, \\
  \end{array}\right.$$
  particularly $(s,0;*;y_1,y_2), (s,1;*;y_1,y_2), (1,t;*;y_1,y_2) \mapsto H(0;*;y_1,y_2)= \theta'_e (*;y_1,y_2)$. This gives a map
  $$0\times (I\times S^l\times (Y\vee Y)) \cup I\times ((\partial I\times S^l\times (Y\vee Y)) \cup (I\times *\times (Y\vee Y))) \to Y$$
  which can be extended to a map $H_3: I\times I\times S^l\times (Y\vee Y)\to Y$. Let $H'(t;a;y_1,y_2)= H_3(1,t;a;y_1,y_2)$.
\end{proof}

\begin{cor}
  For $\alpha\in \mc{Z}_2 [S^l, \ms{C}]$, $\mu'_2- \theta_{\alpha}: S^l \times Y^2 \to Y$ factors through $S^l \wedge Y \wedge Y$,
  \begin{diagram}
    S^l \times Y \times Y  &\rTo^{\mu'_2- \theta_{\alpha}}  &Y_0 \\
    \dTo  &\ruDashto_{\bar{\theta}_{\alpha}} \\
    S^l \wedge Y \wedge Y.
  \end{diagram}
\end{cor}

Suppose $Y$ is furthermore a group-like space (this holds if $Y$ is a connected CW-complex \cite{Whitehead:1978:EHT}). Then the identity has an inverse, denoted by $-\id: Y\to Y$. Let
$$\theta_{\alpha,-}= \theta_{\alpha} \circ (\id_{S^l} \times (-\id)^k).$$

\begin{cor}
  For $\alpha\in \mc{Z}_2 [S^l, \ms{C}]$, $\mu'_2+ \theta_{\alpha,-}$ and $\theta_{\alpha}+ \theta_{\alpha,-}$ both factors through $S^l \wedge Y \wedge Y$,
  $$
  \begin{diagram}
    S^l \times Y \times Y  &\rTo^{\mu'_2+ \theta_{\alpha,-}}  &Y_0 \\
    \dTo  &\ruDashto_{\bar{\theta}_{\alpha,-}} \\
    S^l \wedge Y \wedge Y,
  \end{diagram} \qquad
  \begin{diagram}
    S^l \times Y \times Y  &\rTo^{\theta_{\alpha}+ \theta_{\alpha,-}}  &Y_0 \\
    \dTo  &\ruDashto_{\bar{\theta}_{\alpha} + \bar{\theta}_{\alpha,-}} \\
    S^l \wedge Y \wedge Y.
  \end{diagram}$$
\end{cor}

Clearly $\bar{\theta}_{\alpha,-}= \bar{\theta}_{\alpha} \circ (\id_{S^l} \wedge (-\id)^{\wedge 2})$.

\begin{prop}
  For primitive $y_1, y_2\in H_*Y$,
  \begin{align*}
    (\bar{\theta}_{\alpha})_* (\iota \otimes y_1 \otimes y_2) & = \theta_* (\bar{\alpha} \otimes y_1 \otimes y_2), \\
    (\bar{\theta}_{\alpha,-})_* (\iota \otimes y_1 \otimes y_2) & = \theta_* (\bar{\alpha} \otimes y_1 \otimes y_2), \\
    (\bar{\theta}_{\alpha}+ \bar{\theta}_{\alpha,-})_* (\iota \otimes y_1 \otimes y_2) & = 2\theta_* (\bar{\alpha} \otimes y_1 \otimes y_2).
  \end{align*}
\end{prop}

\subsection{General Cases}
To generalize the proof for the case of two factors to the general cases of more factors, the key is to modify those homotopies so that they can be glued together. Let us analyze the general situation first.

Suppose $\alpha=[f]\in \mc{Z}_k [S^l, \ms{C}]$, namely $f: S^l\to \ms{C}(k)$ with $d_if\simeq \epsilon_{e_{k-1}}$ for $1\leq i\leq k$. We may assume $f(*)=e_k$ (if not, choose $g\simeq f$ with $g(*)=e_k$, then $\theta_f\simeq \theta_{g}$). Restricted to $*\times Y^k\subset S^l\times Y^k$, $\theta_f= \theta(e_k;-)= \mu'_k$ ($\theta(e_k;-)$ and $\mu'_k$ are regarded as the same since they are homotopic). Note that for $y=(y_1, \ldots, y_k)\in Y^k$, if $y_i=*$,
$$\theta_f (a; y)= \theta (f(a); y)= \theta (d_if(a);y), \quad \mu'_k(a;y)= \theta (e_k;y)= \theta (d_ie_k;y).$$
Since $d_if\simeq \epsilon_{e_{k-1}} \simeq \epsilon_{d_ie_k}$, there is $F_i: I\times S^l\to \ms{C}(k-1)$ such that $F_i(0;a)= d_if(a)$, $F_i(1;a)= d_ie_k$. Define
$$F'_i: I\times S^l\times Y_i\to \ms{C}(k-1)\times Y^{k-1}\to Y, \quad F'_i(t;a;y)= \theta (F_i(t;a);d_iy),$$
where $Y_i= Y^{i-1}\times *\times Y^{k-i}$. For $i<j$ and $y\in Y_i\cap Y_j$,
\begin{align*}
  & F'_i(t;a;y)= \theta (F_i(t;a);d_iy)= \theta (d_{j-1}F_i(t;a); d_{j-1}d_iy), \\
  & F'_j(t;a;y)= \theta (F_j(t;a);d_jy)= \theta (d_iF_j(t;a); d_id_jy).
\end{align*}
Note that $d_{j-1}d_iy= d_id_jy$,
\begin{align*}
  & d_{j-1}F_i(0;a)= d_{j-1}d_i f(a)= d_id_j f(a)= d_iF_j(0;a), \\
  & d_{j-1}F_i(1;a)= d_{j-1}d_i e_k= d_id_j e_k= d_iF_j(1;a).
\end{align*}
For $0<t<1$, we have $d_{j-1}F_i(t;a) \sim d_iF_j(t;a)$, but may not have $d_{j-1}F_i(t;a)= d_iF_j(t;a)$. If we only have $d_{j-1}F_i(t;a) \sim d_iF_j(t;a)$, there is no obvious way to modify those homotopies so that they can be glued together. Nevertheless, if $d_{j-1}F_i(t;a)= d_iF_j(t;a)$ for all $a\in S^l$, $0\leq t\leq 1$ and $1\leq i<j\leq k$, then all $F'_i$ can be glued together to give a homotopy from $\theta_f$ to $\mu'_k$ restricted to $I\times S^l\times \mr{FW}(Y^k)$, which can be modified to a homotopy from $\theta_f$ to $\mu'_k$ restricted to $I\times \mr{FW}(S^l\times Y^k)$ using the same trick.

\begin{lem}
  Let $\ms{C}$ be an operad with a basepoint and $k\geq 3$. If $\alpha\in \mc{Z}_k [S^l, \ms{C}]$ ($l\geq 1$) has a representative map $f: S^l\to \ms{C}(k)$ with $f(*)=e_k$ and homotopies $d_if \stackrel{F_i}{\simeq} \epsilon_{d_ie_k}: S^l\to \ms{C}(k-1)$ such that $d_{j-1}F_i(t;a)= d_iF_j(t;a)$ for all $a\in S^l$, $0\leq t\leq 1$ and $1\leq i<j\leq k$, then
  $$\theta_{\alpha} \simeq \mu'_k: \mr{FW}(S^l\times Y^k) \to Y.$$
  In particular, the assertion holds if $\alpha$ has a representative $f$ with $d_if= \epsilon_{e_{k-1}}$ for $1\leq i\leq k$. \qed
\end{lem}

Though the condition in this lemma is strong, there is still a family of operads satisfying such condition. In the following, we shall need the theory of group operads and topological $K(\pi,1)$ operads developed in \cite{Zhang:2011:GOHT}.

Let $\ms{G}= \{G_k\}_{k\geq 0}$ be a group operad with $\pi: \ms{G}\to \ms{S}$. Denote $\ms{H}= \{H_k\}_{k\geq 0}= \Ker \pi$. $E\ms{G}/ \ms{H}$ is a simplicial operad (symmetric if $\pi$ nontrivial and nonsymmetric if $\pi$ trivial) and $\ms{H}\cong \pi_1 (E\ms{G}/ \ms{H})= \{\pi_1 (EG_k/H_k)\}_{k\geq 0}$ can be canonically embedded in $E\ms{G}/ \ms{H}$ by
$$H_k\xra{\cong} \pi_1 (EG_k/H_k)= \{H_k (e_k,a) \mid a\in H_k\}, \quad a\mapsto H_k(e_k,a).$$
In particular $\mc{Z}_k \pi_1 (E\ms{G}/ \ms{H})$ can be embedded in $E\ms{G}/ \ms{H}$, thus each $a\in \mc{Z}_k \pi_1 (E\ms{G}/ \ms{H})$ can be represented by $H_k (e_k,a)$ which satisfies that $d_i H_k(e_k,a)= H_k (e_k, d_ia)= H_k(e_k,e_k)$ is the identity. Hence we have

\begin{lem}
  For a simplicial set $Y$ with an action of $E\ms{G}/ \ms{H}$ and $a\in \mc{Z}_k \pi_1 (E\ms{G}/ \ms{H})$, $\theta_a= \mu'_k: \mr{FW}(S^l\times Y^k) \to Y$. Taking the geometric realization of $|E\ms{G}/\ms{H}|$, we also have $\theta_{\alpha}= \mu'_k: \mr{FW}(S^l\times Y^k) \to Y$ where $Y$ is a topological $|E\ms{G}/\ms{H}|$-space and $\alpha=|a|\in \mc{Z}_k \pi_1 |E\ms{G}/ \ms{H}|$. \qed
\end{lem}

\begin{lem}[Theorem 5.3 in \cite{Zhang:2011:GOHT}]
  Let $\ms{C}$ be a path-connected, locally path-connected and semilocally simply-connected nonsymmetric operad (resp. symmetric operad with the actions of symmetric groups free) with a good basepoint. If $\ms{C}$ is $K(\pi,1)$, then $\ms{C}\sim |E\pi_1 \ms{C}/ \pi_1 \ms{C}|$ (resp. $\ms{C}\sim |E\pi_1 (\ms{C}/ \ms{S})/ \pi_1 \ms{C}|$). \qed
\end{lem}

\begin{lem}
  Suppose Conjecture \ref{conj:smash_operation} is true for a topological operad $\ms{C}$. 1) If there is an equivalence $\phi: \ms{C}\to \ms{C}'$ of operads, then it is also true for $\ms{C}'$. 2) If there is an equivalence $\psi: \ms{C}'\to \ms{C}$ of operads, then it is true for $\ms{C}'$ and for path-connected $\ms{C}'$-spaces of the homotopy type of CW-complexes. Hence if $\ms{C}\sim \ms{C}'$, then it is true for $\ms{C}'$ and for path-connected $\ms{C}'$-spaces of the homotopy type of CW-complexes.
\end{lem}

\begin{proof}
  Note that an equivalence of operads induces isomorphisms on homotopy groups. If there is an equivalence $\phi: \ms{C}\to \ms{C}'$ of operads, then any $\ms{C}'$-space $Y$ is also a $\ms{C}$-space and any $f: S^l \to \ms{C}'(k)$ can be factored as $S^l \to \ms{C}(k) \xra{\phi} \ms{C}'(k)$ so that Conjecture \ref{conj:smash_operation} is also true for $\ms{C}'$. 
  
  Suppose $\psi: \ms{C}'\to \ms{C}$ is an equivalence, then $\psi: \ms{C}'X\to \ms{C}X$ is a homotopy equivalence where $X$ is a path-connected pointed space of the homotopy type of a CW-complex \cite{May:1972:GILS}. Given $f: S^l\to \ms{C}'(k)$, we have the following commutative diagram
  \begin{diagram}
    S^l \times (\ms{C}'X)^k  &\rTo^{f\times \id^k}  &\ms{C}'(k) \times (\ms{C}'X)^k  &\rTo  &\ms{C}'X \\
    \dTo  &&\dTo  &&\dTo \\
    S^l \times (\ms{C}'X)^k  &\rTo^{(\psi \circ f) \times \id^k} &\ms{C}(k) \times (\ms{C}X)^k  &\rTo  &\ms{C}X.
  \end{diagram}
  If $[f]\in \mc{Z}_k [S^l, \ms{C}']$, then $[\psi\circ f]\in \mc{Z}_k [S^l, \ms{C}]$, thus the following diagram
  \begin{diagram}
    \mr{FW}(S^l \times (\ms{C}X)^k)  &\rInto  &S^l \times (\ms{C}X)^k \\
    \dTo  &&\dTo>{\theta_{\psi\circ f}} \\
    (\ms{C}X)^k  &\rTo^{\mu_k}  &\ms{C}X
  \end{diagram}
  is homotopy commutative. Consider
  $$\xymatrix@=1.2pc{
  & \mr{FW}(S^l \times (\ms{C}X)^k) \ar@{^(->}[rr] \ar'[d][dd]
      &  & S^l \times (\ms{C}X)^k \ar[dd]        \\
  \mr{FW}(S^l \times (\ms{C}'X)^k) \ar[ur] \ar@{^(->}[rr] \ar[dd]
      &  & S^l \times (\ms{C}'X)^k \ar[ur]\ar[dd] \\
  & (\ms{C}X)^k \ar'[r][rr]
      &  & \ms{C}X                \\
  (\ms{C}'X)^k \ar[rr]\ar[ur]
      &  & \ms{C}'X \ar[ur]_{\simeq}        }$$
  Except the front square, all the other squares are either commutative or homotopy commutative. Hence $\theta_f \simeq \mu_k: {FW}(S^l \times (\ms{C}'X)^k) \to \ms{C}'X$ since $\ms{C}'X \to \ms{C}X$ is a homotopy equivalence. If $Y$ is a path-connected $\ms{C}'$-space, we have the following commutative diagram
  $$\xymatrix{
    S^l \times Y^k \ar[d] \ar[r] &\ms{C}'(k) \times Y^k \ar[d] \ar[r] & Y \\
    S^l \times (\ms{C}'Y)^k \ar[r] &\ms{C}'(k) \times (\ms{C}'Y)^k \ar[r] & \ms{C}'Y \ar[u]  }$$
  Then from the following commutative diagram
  $$\xymatrix@=1.2pc{
  & \mr{FW}(S^l \times (\ms{C}'Y)^k) \ar@{^(->}[rr] \ar'[d][dd]
      &  & S^l \times (\ms{C}'Y)^k \ar[dd]        \\
  \mr{FW}(S^l \times Y^k) \ar[ur] \ar@{^(->}[rr] \ar[dd]
      &  & S^l \times Y^k \ar[ur]\ar[dd] \\
  & (\ms{C}'Y)^k \ar'[r][rr]
      &  & \ms{C}'Y \ar[dl]               \\
  Y^k \ar[rr]\ar[ur]
      &  & Y        }$$
  we also have $\theta_f \simeq \mu_k: {FW}(S^l \times Y^k) \to Y$.
\end{proof}

If a $\ms{C}$-space $Y$ is not path-connected, then we can consider the path-connected component $Y_0$ of the basepoint since $Y_0$ is also a $\ms{C}$-space. Let $\om^n_0 X$ denote the path-connected component of basepoint of $\om^n X$.

\begin{thm}\label{thm:smash_operations-K(pi,1)}
  Let $\ms{C}$ be a path-connected, locally path-connected and semilocally simply-connected nonsymmetric operad or symmetric operad with the actions of symmetric groups free, and with a good basepoint. If $\ms{C}$ is $K(\pi,1)$, then for $\alpha\in \mc{Z}_k [S^1, \ms{C}]$,
  $$\theta_{\alpha} \simeq \mu'_k: \mr{FW}(S^1\times Y_0^k) \to Y_0$$
  for $k\geq 2$ and $\ms{C}$-spaces $Y$ of the homotopy type of CW-complexes. In particular, for $\alpha\in \mc{Z}_k [S^1, \ms{C}_2]= \mr{Brun}_k/ \mr{ca} (P_k)$,
  $$\theta_{\alpha} \simeq \mu'_k: \mr{FW}(S^1\times (\om^2_0 X)^k) \to \om^2_0 X.$$
\end{thm}

\begin{cor}
  $\mu'_k- \theta_{\alpha}: S^1 \times Y_0^k \to Y_0$ factors through $S^1 \wedge Y_0^{\wedge k}$,
  \begin{diagram}
    S^1 \times Y_0^k  &\rTo^{\mu'_k- \theta_{\alpha}}  &Y_0 \\
    \dTo  &\ruDashto_{\bar{\theta}_{\alpha}} \\
    S^1 \wedge Y_0^{\wedge k}.
  \end{diagram}
\end{cor}

For $\alpha\in \mc{Z}_k [S^l, \ms{C}]$, let $\theta_{\alpha,-}= \theta_{\alpha} \circ (\id_{S^1} \times (-\id)^k)$.

\begin{cor}
   $\mu'_k+ \theta_{\alpha,-}$ and $\theta_{\alpha}+ \theta_{\alpha,-}$ both factors through $S^1 \wedge Y_0^{\wedge k}$,
  $$
  \begin{diagram}
    S^1 \times Y_0^k  &\rTo^{\mu'_k+ \theta_{\alpha,-}}  &Y_0 \\
    \dTo  &\ruDashto_{\bar{\theta}_{\alpha,-}} \\
    S^1 \wedge Y_0^{\wedge k},
  \end{diagram} \qquad
  \begin{diagram}
    S^1 \times Y_0^k  &\rTo^{\theta_{\alpha}+ \theta_{\alpha,-}}  &Y_0 \\
    \dTo  &\ruDashto_{\bar{\theta}_{\alpha} + \bar{\theta}_{\alpha,-}} \\
    S^1 \wedge Y_0^{\wedge k}.
  \end{diagram}$$
\end{cor}

Clearly $\bar{\theta}_{\alpha,-}= \bar{\theta}_{\alpha} \circ (\id_{S^1} \wedge (-\id)^{\wedge k})$.

It is natural to ask if the result can be extended to include other path-connected components, namely, $\theta_{\alpha} \simeq \mu'_k: \mr{FW}(S^1\times Y^k) \to Y_0$. However this is unknown yet.

\subsection{Relation with the Samelson Product}
Smash operations constructed above may be thought of as general analogues of the Samelson product $[-,-]: \om X \wedge \om X \to \om X$, and the simplest case probably coincide the Samelson product.

Recall that $\ms{C}_n(2)$ ($n\geq 2$) is $S_2$-equivariantly homotopy equivalent to $S^{n-1}$. Thus
$$\mc{Z}_2 [S^{n-1}, \ms{C}_n]= [S^{n-1}, \ms{C}_n(2)]= [S^{n-1}, S^{n-1}]= \pi_{n-1} S^{n-1}= \Z.$$
Let $\iota_{n-1}\in \mc{Z}_2 [S^{n-1}, \ms{C}_n]$ corresponding to $[\id_{S^{n-1}}] \in \pi_{n-1} S^{n-1}$. For convenience, we shall abbreviate $\iota_{n-1}$ to $\iota$. By Theorem \ref{thm:smash_operations-two}, we have the following smash operation
$$\bar{\theta}_{\iota}: S^{n-1} \wedge \om^n X \wedge \om^n X\to \om^n X.$$
We next discuss its relation with the Samelson product.

For $n\geq 2$, let $\phi_n$ be the Samelson product of the canonical inclusion $\si^{n-1}X \into \lsi (\si^{n-1}X)$ with itself \cite{Neisendorfer:2010:AMUHT}, i.e.
$$\phi_n: \si^{n-1} X \wedge \si^{n-1} X \into \lsi (\si^{n-1} X) \wedge \lsi (\si^{n-1} X) \xra{[-,-]} \lsi (\si^{n-1} X),$$
and
\begin{align*}
  \psi_n: &\ \si^{n-1} X \wedge \si^{n-1} X= \si^{2n-2} X\wedge X \into \si^{2n-2} (\lsn X \wedge \lsn X) \\
  \to &\ \si^{n-1} (S^{n-1} \wedge \lsn X\wedge \lsn X) \xra{\si^{n-1} \bar{\theta}_{\iota}} \si^{n-1} \lsn X \xra{\mr{ev}} \lsi (\si^{n-1} X)
\end{align*}
where the last map $\mr{ev}$ is the evaluation map.

\begin{prop}
  $\phi_n$ and $\psi_n$ induce the same homomorphism in homology.
\end{prop}

Before giving the proof, recall the suspension isomorphism
$$\sigma: \wt{H}_kX \to H_{k+1} (\si^{n-1} X) \cong H_{n-1} S^{n-1} \otimes \wt{H}_kX, \quad a\mapsto \sigma a= \iota \otimes a$$
and the suspension homomorphism
$$\sigma: \wt{H}_* (\om X) \xra{\sigma} H_{*+1} (\si \om X) \xra{\mr{ev}_*} H_{*+1} X.$$
Let $i: X\into \lsi X$ be the canonical inclusion. In the following diagram
$$\xymatrix{
  H_* (\om \si X) \ar[r]^-{\sigma} & H_{*+1} (\si \om \si X) \ar@<0.5ex>@{-->}[d]^{\mr{ev}_*}  \\
  \wt{H}_* X \ar@{^(->}[u]^{i_*} \ar[r]^-{\sigma} & H_{*+1} (\si X) \ar@<0.5ex>[u]^{(\si i)_*}  }$$
$\sigma \circ i_*= (\si i)_* \circ \sigma$ by the naturality of $\sigma$ and $\mr{ev}_* \circ \sigma \circ i_*= \sigma$ since the evaluation map $\mr{ev}: \si \om \si X\to \si X$ has a section $\si i: \si X \into \si \om \si X$.

\begin{proof}
  Since the diagonal map $\Delta: \si^{n-1} X\to \si^{n-1} X \times \si^{n-1} X$ is homotopic to the composite $\si^{n-1} X\to \si^{n-1} X \vee \si^{n-1} X \into \si^{n-1} X \times \si^{n-1} X$, $\Delta_*a= a\otimes 1+ 1\otimes a$ for $a\in \wt{H}_* (\si^{n-1} X)$, namely all elements of $\wt{H}_* (\si^{n-1} X)$ ($n\geq 2$) are primitive. According to Lemma 6.3.7 in \cite{Neisendorfer:2010:AMUHT},
  $$(\phi_n)_*: \wt{H}_* (\si^{n-1} X) \otimes \wt{H}_* (\si^{n-1} X) \to H_* (\lsi (\si^{n-1} X)) \cong T \wt{H}_* (\si^{n-1} X),$$
  $$(\phi_n)_* (\sigma a \otimes \sigma b)= \sigma a \cdot \sigma b- (-1)^{|\sigma a||\sigma b|} \sigma b\cdot \sigma a.$$
  On the other hand,
  $$(\psi_n)_*: \wt{H}_* (\si^{n-1} X) \otimes \wt{H}_* (\si^{n-1} X) \to H_* (\lsi (\si^{n-1} X)) \cong T \wt{H}_* (\si^{n-1} X),$$
  \begin{align*}
    (\psi_n)_* (\sigma a \otimes \sigma b) &= (\psi_n)_* (\iota \otimes a \otimes \iota \otimes b)= (\psi_n)_* ((-1)^{|\iota||a|} \iota \otimes \iota \otimes a\otimes b) \\
    &= (-1)^{(n-1)|a|} \mr{ev}_* (\iota \otimes \theta_* (\iota \otimes a\otimes b)) \\
    &= \sigma a \cdot \sigma b- (-1)^{|\sigma a||\sigma b|} \sigma b\cdot \sigma a,
  \end{align*}
  where the last step is obtained by applying Theorem 2 in \cite{Browder:1960:HOLS} $n-1$ times.
\end{proof}

It is then natural to make the following conjecture.

\begin{conj}
  $\phi_n \simeq \psi_n$.
\end{conj}

It was first guessed that $\bar{\theta}_{\iota}$ should be related to the Samelson product right after observing the following similarity. We shall only consider $\lsii$ but the discussion also applies to $\lsn$. The Samelson product induces
\begin{align*}
  S^1 \wedge (\om^2 \si^2 X)^{\wedge 2} & \into \om^2 \si^2 (S^1 \wedge (\om^2 \si^2 X)^{\wedge 2}) \to \om^2 (S^1 \wedge \si^2 (\om^2 \si^2 X)^{\wedge 2}) \\
  & \to \om^2 (S^1 \wedge \si^2 X \wedge \om^2 \si^2 X) \to \om^2 (S^1 \wedge X \wedge \si^2 X) \\
  & \to \om^2 (S^1 \wedge \si X \wedge \si X) \to \om \om \si (\si X \wedge \si X) \\
  & \xra{\om \bar{\phi_2}} \om \om \si \si X= \om^2 \si^2 X
\end{align*}
where $\bar{\phi_2}: \lsi (\si X\wedge \si X) \to \lsi (\si X)$ is the extension of $\phi_2: \si X\wedge \si X \to \lsi (\si X)$. This map has the same domain and codomain of
$$\bar{\theta}_{\iota}: S^1 \wedge (\om^2 \si^2 X)^{\wedge 2} \to \om^2 \si^2 X.$$
If $\phi_2\simeq \psi_2$, then the above two maps are homotopic as well. This similarity can also be observed from that the loop of $\bar{\phi}_2$ is
$$\om^2 \si^3 (X \wedge X)= \om \om \si (\si X \wedge \si X) \xra{\om \bar{\phi_2}} \om \om \si \si X= \om^2 \si^2 X$$
while $\bar{\theta}_{\iota}$ induces
\begin{align*}
  \om^2 \si^3 (X \wedge X) & \into  \om^2 \si^3 (\om^2 \si^2 X\wedge \om^2 \si^2 X) \to \om^2 \si^2 (S^1 \wedge \om^2 \si^2 X\wedge \om^2 \si^2 X) \\
  & \xra{\lsii \bar{\theta}_{\iota}} \om^2 \si^2 (\om^2 \si^2 X) \xra{\om^2 \mr{ev}} \om^2 \si^2 X
\end{align*}
which also have the same domain and codomain. By taking iteration, we can obtain, from the Samelson product and $\bar{\theta}_{\iota}$, two maps $\om^2 \si^{k+1} X^{\wedge k} \to \lsii X$ for $k\geq 2$ which would be homotopic if $\phi_2\simeq \psi_2$.

Moreover, the maps $\om^2 \si^{k+1} X^{\wedge k} \to \lsii X$ ($k\geq 2$) look interesting, as $\om^2 \si^k X^{\wedge k}$ is the target space of $\om h_k$ where $h_k$ is the $k$th James-Hopf invariant \cite{Cohen:preprint:CGTHTI, Cohen:1995:CGTH}. Recall \cite{Cohen:preprint:CGTHTI, Cohen:1995:CGTH} that an important family of self-maps of $\lsi X$ are the following composites
$$\lsi X \xra{h_k} \lsi X^{\wedge k} \xra{\textrm{iteration of Samelson product}} \lsi X.$$
This together with the viewpoint that smash operations are general analogues of the Samelson product suggest that there might exist maps
$$\lsii X \to \lsii (S^l \wedge X^{\wedge k})$$
which would be 2-dimensional analogues of the James-Hopf invariants and that a family of self-maps of $\lsii X$ may be of the following form
$$\lsii X \xra{?} \lsii (S^l \wedge X^{\wedge k}) \xra{\textrm{(mixed) iteration of smash operations}} \lsii X$$
where the first part can be loop of the James-Hopf invariants and could be 2-dimensional analogues of the James-Hopf invariants. For instance, iteration of smash operations
$$S^1 \wedge (\lsii X)^{\wedge 2} \to \lsii X$$
gives maps
$$S^{k-1} \wedge (\lsii X)^{\wedge k} \to \lsii X$$
which provide many self-maps of $\lsii X$
$$\lsii X= \om \lsi (\si X) \xra{\om h_k} \om \lsi (\si X)^{\wedge k}= \lsii (S^{k-1} \wedge X^{\wedge k}) \to \lsii X.$$

The above discussion also applies to $\lsn X$. Just like that the Samelson product $\om X\wedge \om X\to \om X$ can induce maps $\lsi X^{\wedge k} \to \lsi X$, mixed iteration of smash operations
$$S^l \wedge (\om^2 X)^{\wedge k} \to \om^2 X$$
can also induce maps for $n>2$
\begin{align*}
  \lsn (S^l \wedge X^{\wedge k}) &\into \lsn (S^l \wedge (\lsii X)^{\wedge k}) \to \lsn (\lsii X) \\
  &\to \om^n \si^{n-2} \si^2 X= \lsn X
\end{align*}
which should also coincide with certain smash operations on $\lsn X$ constructed in the same way of those on $\lsii X$. Then similarly there might exist maps
$$\lsn X \to \lsn (S^l \wedge X^{\wedge k})$$
which would be $n$-dimensional analogues of the James-Hopf invariants and a family of self-maps of $\lsn X$ may be of the following form
$$\lsn X \xra{?} \lsn (S^l \wedge X^{\wedge k}) \xra{\textrm{(mixed) iteration of smash operations}} \lsii X$$
where the first part can be loop of the James-Hopf invariants and could be $n$-dimensional analogues of the James-Hopf invariants.

\section{Applications to Homotopy Groups}
An application of smash operations is that they induce operations on homotopy groups which can be assembled together to give a conceptual description of the structure of homotopy groups.

Generally, any map
$$\phi: X_1 \wedge \cdots \wedge X_k \to X,$$
induces a family of functions on homotopy groups
$$\phi_*: \pi_{m_1}X_1 \times \cdots \times \pi_{m_k}X_k \to \pi_mX, \quad m_i\geq 0,$$
sending $[f_i]\in \pi_{m_i}X_i$, $1\leq i\leq k$, to the homotopy class of the composite
$$S^m= S^{m_1} \wedge \cdots \wedge S^{m_k} \xra{f_1 \wedge \cdots \wedge f_k} X_1 \wedge \cdots \wedge X_k \xra{\phi} X.$$
Therefore, each smash operation $\bar{\theta}_{\alpha}$ (if exists) on $\ms{C}$-spaces $Y$ induces a family of functions on homotopy groups
$$(\bar{\theta}_{\alpha})_*: \pi_lS^l \times \pi_{m_1}Y \times \cdots \times \pi_{m_k}Y\to \pi_mY,$$
where $\alpha\in \mc{Z}_k [S^l, \ms{C}]$, which are natural with respect to $\ms{C}$-maps. Thus call them operations on homotopy groups.

$(\bar{\theta}_{\iota})_*$ where $\iota$ is the fundamental class of $\pi_{n-1}\ms{C}_n(2)$, is related to the Whitehead product $[-,-]$. Let $\om^n_0 X$ denote the path-connected component of the basepoint of $\om^n X$.

\subsection{Smash Product on Homotopy Groups}
Let $X_1,\ldots,X_k$ ($k\geq 2$) be pointed spaces. Define the smash product on homotopy groups
\begin{align*}
  \wedge: \pi_{m_1}X_1 \times \cdots \times \pi_{m_k}X_k & \to \pi_{m} (X_1\wedge \cdots \wedge X_k) \\
  ([f_1], \ldots, [f_k]) & \mapsto [f_1 \wedge \cdots \wedge f_k],
\end{align*}
where $m_i\geq 0$.

The notions of bilinear homomorphism and multilinear homomorphism can be extended to general monoids and groups (unnecessary abelian). For nonabelian monoids and groups, the product is still denoted by $+$.

\begin{lem}
  The smash product $\pi_{m_1}X_1 \times \cdots \times \pi_{m_k}X_k \to \pi_{m} (X_1\wedge \cdots \wedge X_k)$ is linear on the $i$th factor if $m_i\geq 1$.
\end{lem}

\begin{proof}
  Consider the case $k=2$ first. Let $[f_1], [f_1']\in \pi_iX_1$, $[f_2]\in \pi_jX_2$ and $i\geq 1$, then $([f_1]+ [f_1']) \wedge [f_2]$ is represented by
  $$S^{i+j}= S^i\wedge S^j\to (S^i\vee S^i)\wedge S^j \xra{(f_1\vee f_1') \wedge f_2} (X_1 \vee X_1)\wedge X_j \to X_1\wedge X_2$$
  and $[f_1]\wedge [f_2]+ [f_1']\wedge [f_2]$ is represented by
  $$S^{i+j}\to S^{i+j} \vee S^{i+j} \to (S^i\wedge S^j) \vee (S^i \wedge S^j) \xra{(f_1\wedge f_2) \vee (f_1'\wedge f_2)} (X_1\wedge X_2) \vee (X_1\wedge X_2) \to X_1\wedge X_2.$$
  It is clear that the following diagram is commutative
  \begin{diagram}
    S^{i+j}  &\rEq  &S^i\wedge S^j  &\rTo  &(S^i\vee S^i)\wedge S^j  &\rTo^{(f_1\vee f_1') \wedge f_2}  &(X_1 \vee X_1)\wedge X_2  &\rTo  &X_1\wedge X_2 \\
    &\rdTo  &&&\dTo>{\cong}  &&\dTo>{\cong}  &\ruTo \\
    &&S^{i+j} \vee S^{i+j}  &\rTo  &(S^i\wedge S^j)^{\vee 2}  &\rTo^{(f_1\wedge f_2) \vee (f_1'\wedge f_2)} &(X_1\wedge X_2)^{\vee 2}.
  \end{diagram}
  So
  $$([f_1]+ [f_1']) \wedge [f_2]= [f_1]\wedge [f_2]+ [f_1']\wedge [f_2].$$
  Similarly the smash product is linear on the second factor if $j\geq 1$. The cases $k>2$ follows by induction from the following decomposition of the smash product
  $$\pi_{m_1}X_1 \times \cdots \times \pi_{m_k}X_k \to \pi_{m_1+ \cdots+ m_{k-1}} (X_1 \wedge \cdots \wedge X_{k-1}) \times \pi_{m_k}X_k \to \pi_{m} (X_1\wedge \cdots \wedge X_k).$$
\end{proof}

It should be noted that if $m_i=0$, the smash product is generally not linear on the $i$th factor.

\begin{prop}
  For any pointed map $\phi: X_1 \wedge \cdots \wedge X_k \to X$, the composite
  $$\pi_{m_1}X_1 \times \cdots \times \pi_{m_k}X_k \xra{\wedge} \pi_m(X_1\wedge \cdots\wedge X_k) \xra{\phi_*} \pi_mX,$$
  is linear on the $i$th factor if $m_i\geq 1$. It is particularly multilinear if all $m_i\geq 1$. \qed
\end{prop}

Recall $S^0=\{1,-1\}$ with basepoint 1. Choose a generator of $\wt{H}_0S^0= H_0(-1)\subset H_0S^0$ and denote it $\iota_0$. Let $\iota_k$ be the generator of $\wt{H}_kS^k= H_kS^k$ ($k\geq 1$) via the suspension isomorphism $\wt{H}_0S^0 \cong \wt{H}_1S^1\cong \wt{H}_2S^2 \cong \cdots$. Then from $S^i\wedge S^j\cong S^{i+j}$ ($i,j\geq 0$),
$$\wt{H}_iS^i \times \wt{H}_jS^j \xra{\wedge} \wt{H}_{i+j} (S^i \wedge S^j) \xra{\cong} \wt{H}_{i+j} S^{i+j}, \quad (\iota_i, \iota_j) \mapsto \iota_i \wedge \iota_j \mapsto \iota_{i+j}.$$
For $i\geq 0$, the Hurewicz map is
$$h: \pi_iX \to H_iX \quad [f] \mapsto f_*(\iota_i),$$
where $f_*: H_iS^i\to H_iX$ is the induced homomorphism of the pointed map $f: S^i\to X$. Note that $h: \pi_0X \to H_0X= \Z[\pi_0X]$ sends $a\in \pi_0X$ to $a\in \pi_0X \subset \Z[\pi_0X]$ and thus is not a homomorphism! For instance,
$$h: \Z= \pi_0 \om^n S^n \to H_0 \om^n S^n= \Z[\Z]= \Z[\ldots, t^{-2}, t^{-1}, 1, t, t^2, \ldots]$$
sends $i$ to $t^i$ ($n\geq 1$).

\begin{lem}
  The following diagram is commutative for $m_i\geq 0$
  \begin{diagram}
    \pi_{m_1}X_1 \times \cdots \times \pi_{m_k}X_k  &\rTo^{\wedge}  &\pi_m(X_1\wedge \cdots\wedge X_k) \\
    \dTo<{h\times \cdots \times h}  &&\dTo>h \\
    H_{m_1}X_1 \times \cdots \times H_{m_k}X_k  &\rTo^{\wedge}  &H_m(X_1\wedge \cdots\wedge X_k).
  \end{diagram}
\end{lem}

\begin{proof}
  It suffices to consider the case $k=2$. For $i,j\geq 0$, $[f]\in \pi_iX$ and $[g]\in \pi_jY$, the following diagram is commutative
  \begin{diagram}
    H_iS^i\times H_jS^j  &\rTo  &H_{i+j} (S^i\times S^j)  &\rTo &H_{i+j} (S^i\wedge S^j) \\
    \dTo<{f_*\times g_*}  &&\dTo>{(f\times g)_*}  &&\dTo>{(f\wedge g)_*} \\
    H_iX\times H_jY  &\rTo  &H_{i+j} (X\times Y)  &\rTo &H_{i+j} (X\wedge Y).
  \end{diagram}
  Thus $(f\wedge g)_* (\iota_i \wedge \iota_j)= f_* \iota_i \wedge g_* \iota_j$. Then
  $$h([f]\wedge [g])= h[f\wedge g]= (f\wedge g)_* (\iota_{i+j})= (f\wedge g)_* (\iota_i \wedge \iota_j)= f_* \iota_i \wedge g_* \iota_j= h[f] \wedge h[g];$$
  namely the following diagram
  \begin{diagram}
    \pi_iX \times \pi_j Y  &\rTo^{\wedge}  &\pi_{i+j} (X\wedge Y) \\
    \dTo<{h\times h}  &&\dTo>h \\
    H_iX\times H_jY  &\rTo^{\wedge}  &H_{i+j} (X\wedge Y).
  \end{diagram}
  is commutative.
\end{proof}

\begin{prop}
  For any pointed map $\phi: X_1 \wedge \cdots \wedge X_k\to X$, the following diagram is commutative for $m_i\geq 0$
  \begin{diagram}
    \pi_{m_1}X_1 \times \cdots \times \pi_{m_k}X_k  &\rTo^{\wedge}  &\pi_m(X_1\wedge \cdots\wedge X_k)  &\rTo{\phi_*}  &\pi_mX \\
    \dTo<{h\times \cdots \times h}  &&\dTo>h  &&\dTo>h \\
    H_{m_1}X_1 \times \cdots \times H_{m_k}X_k  &\rTo^{\wedge}  &H_m(X_1\wedge \cdots\wedge X_k)  &\rTo{\phi_*}  &H_mX.
  \end{diagram}
\end{prop}

\begin{proof}
  Commutativity of the second square follows from the naturality of the Hurewicz homomorphism.
\end{proof}

\subsection{Induced Operations on Homotopy Groups}
For $\alpha\in \mc{Z}_k [S^l, \ms{C}]$ ($l\geq 1$), if $\theta_{\alpha}\simeq \mu'_k$ restricted to the fat wedge $\mr{FW} (S^l\times Y^k)$, then we have a smash operation $\bar{\theta}_{\alpha}: S^l\wedge Y^{\wedge k}\to Y$ inducing
$$(\bar{\theta}_{\alpha})_*: \pi_{l'}S^l \times \pi_{m_1}Y \times \cdots \times \pi_{m_k}Y\to \pi_{l'+m}Y$$
which is linear on $\pi_{l'}S^l$ for $l'\geq l$ and linear on $\pi_{m_i}Y$ for $m_i\geq 1$. For $[\phi]\in \pi_{l'}S^l$, we have the following commutative diagram
\begin{diagram}
  \pi_{m_1}Y \times \cdots \times \pi_{m_k}Y  &\rTo^{(\bar{\theta}_{\alpha})_* ([\phi];-)} &\pi_{l'+m}Y \\
  &\rdTo_{(\bar{\theta}_{\alpha})_* (\iota;-)}  &\uTo>{\bar{\phi}} \\
  &&\pi_{l+m}Y
\end{diagram}
where $\bar{\phi}: \pi_{l+m}Y\to \pi_{l'+m}Y$, $[f] \mapsto [f\circ (\phi\wedge \id_{S^m})]$, following from the obvious commutative diagram
\begin{diagram}
  S^{l'}\wedge S^{m_1}\wedge \cdots \wedge S^{m_k}  &\rTo  &S^l \wedge Y^{\wedge k}  &\rTo^{\bar{\theta}_{\alpha}}  &Y \\
  \dTo<{\phi\wedge \id}  &\ruTo \\
  S^l\wedge S^{m_1}\wedge \cdots \wedge S^{m_k}.
\end{diagram}
So we shall only consider the case $l'=l$. Since $(\bar{\theta}_{\alpha})_*$ is linear on $\pi_lS^l$, we need only consider
$$\wt{\theta}_{\alpha}:= (\bar{\theta}_{\alpha})_* (\iota;-): \pi_{m_1}Y \times \cdots \times \pi_{m_k}Y \to \pi_{l+m}Y$$
where $\iota= [\id_{S^l}]\in \pi_lS^l$. Also use the same notation to denote
$$\wt{\theta}_{\alpha}:= (\bar{\theta}_{\alpha})_* (\iota;-): H_{m_1}Y \times \cdots \times H_{m_k}Y \to H_{l+m}Y.$$

\begin{prop}
  Under the condition of Theorem \ref{thm:smash_operations-two}, for $\alpha\in \mc{Z}_2 [S^l, \ms{C}]$, $\wt{\theta}_{\alpha}: \pi_{m_1} Y \times \pi_{m_2} Y \to \pi_{l+m_1+m_2}Y$ is linear on $\pi_{m_i}Y$ if $m_i\geq 1$ and the following diagram is commutative
  \begin{diagram}
    \pi_{m_1}Y \times \pi_{m_2}Y  &\rTo^{\wt{\theta}_{\alpha}} &\pi_{l+m_1+m_2}Y \\
    \dTo<h  &&\dTo>h \\
    H_{m_1}Y \times H_{m_2}Y  &\rTo^{\wt{\theta}_{\alpha}} &H_{l+m_1+m_2}Y
  \end{diagram}
  for $m_i\geq 0$. \qed
\end{prop}

\begin{prop}
  Under the condition of Theorem \ref{thm:smash_operations-K(pi,1)}, for $\alpha\in \mc{Z}_k [S^1, \ms{C}]$, $\wt{\theta}_{\alpha}: \pi_{m_1}Y_0 \times \cdots \times \pi_{m_k}Y_0\to \pi_{1+m}Y_0$ is linear on $\pi_{m_i}Y$ if $m_i\geq 1$ and the following diagram is commutative
  \begin{diagram}
    \pi_{m_1}Y_0 \times \cdots \times \pi_{m_k}Y_0  &\rTo^{\wt{\theta}_{\alpha}} &\pi_{1+m}Y_0 \\
    \dTo<h  &&\dTo>h \\
    H_{m_1}Y_0 \times \cdots \times H_{m_k}Y_0  &\rTo^{\wt{\theta}_{\alpha}} &H_{1+m}Y_0
  \end{diagram}
  for $m_i\geq 0$. \qed
\end{prop}

\begin{example}
  The Samelson product $[-,-]: \om X\wedge \om X\to \om X$ induces a product on homotopy groups
  $$[-,-]: \pi_i \om X \times \pi_j \om X\to \pi_{i+j} \om X.$$
  For $n\geq 2$, $\bar{\theta}_{\iota}= \bar{\theta}_{\iota_{n-1}}: S^{n-1} \wedge \om^n X \wedge \om^n X \to \om^n X$ induces a product
  $$\wt{\theta}_{\iota}: \pi_i \om^n X \times \pi_j \om^n X\to \pi_{i+j} \om^n X.$$
  For $\alpha\in \Brun_k/ \mr{ca}(P_k)$, $\bar{\theta}_{\alpha}: S^1 \wedge (\om^2_0 X)^{\wedge k} \to \om^2_0 X$ induces
  $$\wt{\theta}_{\alpha}: \pi_{m_1} \om^2_0 X \times \cdots \times \pi_{m_k} \om^2_0 X \to \pi_{1+m} \om^2_0 X,$$
  namely
  $$\wt{\theta}_{\alpha}: \pi_{2+m_1} X \times \cdots \times \pi_{2+m_k} X \to \pi_{3+m} X$$
  for $m_i\geq 1$. In particular when $n\geq 3$, $X= S^n$ and $2+m_i= n$,
  $$\wt{\theta}_{\alpha}: \pi_n S^n \times \cdots \times \pi_n S^n \to \pi_{3+k(n-2)} S^n.$$
  Clearly each $\alpha\in \Brun_k/ \mr{ca}(P_k)$ gives an element $\wt{\theta}_{\alpha} (\iota_n^{(k)})\in \pi_{3+k(n-2)} S^n$. Thus the identity map of $S^n$ generates a family of elements in $\pi_* S^n$ under all operations $\wt{\theta}_{\alpha}$ and there will be more if take (mixed) iterations of these operations. If $k\geq 3$, $\wt{\theta}_{\alpha}$ is unfortunately trivial in homology since Brunnian braids are commutators, so that information of $\wt{\theta}_{\alpha}$ on homotopy groups can not be obtained from homology.
\end{example}

Corresponding to the relation between $\bar{\theta}_{\iota}$ and the Samelson product, the induced operation $\wt{\theta}_{\iota}$ on homotopy groups is related to the Whitehead product
$$[-,-]: \pi_i \om^n X \times \pi_j \om^n X= \pi_{n+i} X \times \pi_{n+j} X \to \pi_{2n-1+i+j} X.$$

\begin{prop}
  For $x\in \pi_i \om^n X$ and $y\in \pi_j \om^n X$ ($i,j\geq 1$),
  $$\wt{\theta}_{\iota}(x,y)- [x,y]\in \Ker (h: \pi_{n-1+i+j} \om^n_0 X \to H_{n-1+i+j} \om^n_0 X),$$
  where $h$ is the Hurewicz homomorphism.
\end{prop}

This proposition is an immediate consequence of the above discussion and the following lemma.

\begin{lem}[Cohen (1976) \cite{Coh-Lad-May:1976:HILS}]
  The following diagram is commutative for $i,j\geq 1$
  \begin{diagram}
    \pi_i \om^n X \times \pi_j \om^n X  &\rTo^{[-,-]}  &\pi_{n-1+i+j} \om^n X \\
    \dTo  &&\dTo\\
    H_i \om^n X \times H_j \om^n X  &\rTo^{\lambda_{n-1}}  &H_{n-1+i+j} \om^n X
  \end{diagram}
  where the Browder operation $\lambda_{n-1}= (-1)^{(n-1)i+1} \theta_* (\iota_{n-1} \otimes a \otimes b)$. \qed
\end{lem}

It should be noted that this lemma does not hold generally if $i=0$ or $j=0$. Noting $\theta: \ms{C}_2(2) \times \om^n_iS^n \times \om^n_jS^n \to \om^n_{i+j}S^n$, $\lambda_{n-1} (\iota_n, \iota_n)\in H_{n-1} \om^n_2S^n$, while $h[\iota_n, \iota_n]\in h\pi_{n-1} \om^nS^n= h\pi_{n-1} \om^n_0S^n \subseteq H_{n-1} \om^n_0S^n$. It seems they would become the same after translating $\lambda_{n-1} (\iota_n, \iota_n)$ into $H_{n-1} \om^n_0S^n$.

It is also natural to make the following conjecture.

\begin{conj}
  $\wt{\theta}_{\iota}$ coincides with the Whitehead product $[-,-]$ on homotopy groups, namely the following diagram is commutative
  $$\xymatrix{
    \pi_{n+i} X \times \pi_{n+j} X \ar@{=}[d] \ar[r]^-{[-,-]} & \pi_{2n+i+j-1} X \ar@{=}[d] \\
    \pi_i \om^n X \times \pi_j \om^n X \ar[r]^-{\wt{\theta}_{\iota}} & \pi_{n-1+i+j} \om^n X   }$$
\end{conj}

Composed with the canonical inclusion $X\into \lsn X$, smash operations also provide operations connecting the homotopy groups of various suspensions of a space. Namely, for a path-connected pointed space $X$, if there is a map $S^l \wedge (\om^n X)^{\wedge k} \to \om^n X$, then
$$S^l \wedge X^{\wedge k} \into S^l \wedge (\lsn X)^{\wedge k} \to \lsn X$$
induces a multilinear homomorphism
$$\pi_{m_1} X \times \cdots \times \pi_{m_k} X \to \pi_{m_1} \lsn X \times \cdots \times \pi_{m_k} \lsn X \to \pi_{l+m} \lsn X= \pi_{l+n+m} \si^n X$$
for $m_i\geq 0$. Thus the Samelson product
$$X\wedge X \into \lsi X\wedge \lsi X \xra{[-,-]} \lsi X$$
induces a bilinear homomorphism
$$\pi_i X \times \pi_j X \to \pi_i \lsi X \times \pi_j \lsi X \to \pi_{i+j} \lsi X = \pi_{1+i+j} \si X,$$
and
$$S^{n-1} \wedge X \wedge X \into S^{n-1} \wedge \lsn X \wedge \lsn X \xra{\bar{\theta}_{\iota}} \lsn X$$
also induces a bilinear homomorphism
$$\pi_i X \times \pi_j X \into \pi_i \lsn X \times \pi_j \lsn X \to \pi_{n-1+i+j} \lsn X= \pi_{2n-1+i+j} \si^n X.$$
Similarly, for $\alpha\in \Brun_k/ \mr{ca}(P_k)$, $\bar{\theta}_{\alpha}: S^1 \wedge (\om^2_0 X)^{\wedge k} \to \om^2_0 X$ gives
$$S^1 \wedge X^{\wedge k} \into S^1 \wedge (\lsii X)^{\wedge k} \xra{\bar{\theta}_{\alpha}} \lsii X$$
inducing a multilinear homomorphism
$$\pi_{m_1} X \times \cdots \times \pi_{m_k} X \to \pi_{m_1} \lsii X \times \cdots \times \pi_{m_k} \lsii X \to \pi_{1+m} \lsii X= \pi_{3+m} \si^2 X.$$
If smash operations $S^l \wedge (\om^n X)^{\wedge k} \to \om^n X$ could also be established for $n,k\geq 3$, then they would induce multilinear homomorphisms
$$\pi_{m_1} X \times \cdots \times \pi_{m_k} X \to \pi_{l+n+m} \si^n X.$$

It is of particular interest to consider this type of operations on the homotopy groups for various spheres $S^n$, $n\geq 1$. The Samelson product on $\om X$ induces a binary operation from $\pi_* S^n$ to $\pi_* S^{n+1}$, smash operations on $\om^2 X$ induce operations from $\pi_* S^n$ to $\pi_* S^{n+2}$, and smash operations on $\om^m X$ induce operations from $\pi_* S^n$ to $\pi_* S^{n+m}$. Hence all the homotopy groups of all spheres $\{\pi_m S^n\}_{m,n\geq 1}$ can be connected by these operations.

Smash operations are also connected to the homotopy groups of wedges of spheres. Recall \cite{Whitehead:1978:EHT} that the set of all operations on homotopy groups is in one-to-one correspondence with homotopy groups of wedges of spheres. Let $X= S^{2+m_1} \vee \cdots \vee S^{2+m_k}$, and $\iota_{2+m_i}$ the homotopy class of the inclusion $S^{2+m_i} \into X$. Hence we have the following functions from the conjugacy classes of Brunnian braids to the homotopy groups of certain wedges of spheres
\begin{align*}
  \wt{\Theta}_2^2: P_2 & \to \pi_{3+i+j} (S^{2+i} \vee S^{2+j}) \\
  \alpha & \mapsto \wt{\theta}_{\alpha} (\iota_{2+i}, \iota_{2+j})
\end{align*}
where $i,j\geq 0$, and
\begin{align*}
  \wt{\Theta}_2^k: \mr{Brun}_k/ P_k & \to \pi_{3+m} (S^{2+m_1} \vee \cdots \vee S^{2+m_k}) \\
  \alpha & \mapsto \wt{\theta}_{\alpha} (\iota_{2+m_1}, \ldots, \iota_{2+m_k})
\end{align*}
where $m_1,\ldots,m_k\geq 1$.

\subsection{Structures of Homotopy Groups}
All these induced smash operations on homotopy groups can be assembled together to give a conceptual description of the structure of homotopy groups. We shall deal with algebraic operads in the following. One may refer to \cite{LodVal:AO} for algebraic operads.

First let us look $\pi_* (\om_0 X)$ from the point of view of algebraic operads. The Samelson product on homotopy groups
$$\theta_{\tau}= [-,-]: \pi_i \om_0 X \times \pi_j \om_0 X \to \pi_{i+j} \om_0 X$$
generates a free algebraic operad such that $\pi_* \om_0 X$ is a module over this algebraic operad with kernel essentially captured by the following two relations \cite{Cohen:1987:CHT}
$$[a,b]= (-1)^{ij+1} [b,a], \quad [[a,b],c]- [a,[b,c]]+ (-1)^{ij} [b,[a,c]]=0.$$

Under the conditions of Theorem \ref{thm:smash_operations-K(pi,1)} and if $\ms{C}$ is symmetric, the action of $S_k$ on $[S^1, \ms{C}(k)]$ restricts to an action of $S_k$ on $\mc{Z}_k [S^1, \ms{C}]$. Let $\ms{F} (\mc{Z} [S^1,\ms{C}]; \Z)$ be the free symmetric algebraic operad over $\Z[\ms{S}]= \{\Z[S_k]\}_{k\geq 0}$ generated by $\mc{Z} [S^1, \ms{C}]= \{\mc{Z}_k [S^1, \ms{C}]\}_{k\geq 1}$. Note that $\mc{Z}_k [S^1, \ms{C}_2]= \Brun_k/ \mr{ca}(P_k)$. (One may also consider nonsymmetric $\ms{C}$.)

\begin{thm}\label{thm:structure_homotopy_groups-K(pi,1)}
  Under the conditions of Theorem \ref{thm:smash_operations-K(pi,1)} and if $\ms{C}$ is symmetric, $\pi_* Y_0$ is a module over $\ms{F} (\mc{Z} [S^1, \ms{C}]; \Z)$, that is $\pi_* Y_0$ admits a natural action of the operad $\ms{F} (\mc{Z} [S^1, \ms{C}]; \Z)$. In particular, $\pi_* \om^2_0 X$ is a module over $\ms{F} (\mc{Z} [S^1, \ms{C}_2]; \Z)$. \qed
\end{thm}

If $\Z$ is replaced by $\F= \Z/p$ or $\Z_{(p)}$, then $\pi_* (Y_0, \F)$ is an algebra over $\ms{F} (\mc{Z} [S^1, \ms{C}]; \F)$ which is an algebraic operad over $\F[\ms{S}]= \{\F[S_k]\}_{k\geq 0}$.

The identity map of $S^n$ ($n\geq 3$) particularly generates a family of elements in $\pi_*S^n$ under the action of $\ms{F} (\mc{Z} [S^1, \ms{C}_2];\Z)$. If the smash operations on $\om^2_0 X$ could be extended to $\om^2 X$ so that $\pi_0 \om^2 X= \pi_2X$ could be included, then the identity map of $S^2$ would generate a family of elements in $\pi_* S^2$ as well. If Conjecture \ref{conj:smash_operation} could be proved, then the results on $\pi_* \om^2_0 X$ could be generalized to $\pi_* \om^n X$ for $n\geq 3$.

Compare with $\pi_* \om_0X$, the kernel of the action of $\ms{F} (\mc{Z} [S^1, \ms{C}_2]; \Z)$ on $\pi_* \om^2_0 X$ would be generated by certain relations analogous to the relations of the Samelson product $[-,-]: \pi_i \om X \times \pi_j \om X \to \pi_{i+j} \om X$ \cite{Cohen:1987:CHT},
  $$[a,b]= (-1)^{ij+1} [b,a], \quad [[a,b],c]- [a,[b,c]]+ (-1)^{ij} [b,[a,c]]=0.$$

\begin{rem}\label{rem:Brunnian-Lie}
  Note that the conjugacy classes of Brunnian braids play an essential role in the structure of the homotopy groups of double loop spaces by Theorem \ref{thm:structure_homotopy_groups-K(pi,1)}. It is then interesting to see that they are related to $\mr{Lie}(n)$ due to Li and Wu \cite{LiWu:preprint:BGBBBHG}. The conjugation action of $P_n$ on $\Brun_n$ factors through the abelianization $\Brun_n^{\mr{ab}}$. Namely the projection
  $$\Brun_n \onto \Brun_n^{\mr{ab}}$$
  induces a function
  $$\Brun_n/ \mr{ca}(P_n) \onto \Brun_n^{\mr{ab}}/ \mr{ca}(P_n)= \Brun_n^{\mr{ab}} \otimes_{\Z[P_n]} \Z.$$
  Li and Wu in \cite{LiWu:preprint:BGBBBHG} show that there is an epimorphism of abelian groups
  $$\Brun_n^{\mr{ab}} \otimes_{\Z[P_n]} \Z \onto \mr{Lie}(n-1)$$
  and $\mr{Lie}(n-1)$ admits an $S_n$-action such that this is a homomorphism of $\Z[S_n]$-modules. Moreover, Wu conjectures that this epimorphism is an isomorphism of $\Z[S_n]$-modules.
\end{rem}

\bibliography{MyReference}

\providecommand{\bysame}{\leavevmode\hbox to3em{\hrulefill}\thinspace}
\providecommand{\MR}{\relax\ifhmode\unskip\space\fi MR }
\providecommand{\MRhref}[2]{%
  \href{http://www.ams.org/mathscinet-getitem?mr=#1}{#2}
}
\providecommand{\href}[2]{#2}
\begin{thebibliography}{10}

\bibitem{Boa-Vog:1968:HEHS}
J.~M. Boardman and R.~M. Vogt, \emph{Homotopy-everything {$H$}-spaces}, Bull.
  Amer. Math. Soc. \textbf{74} (1968), 1117--1122.

\bibitem{Boa-Vog:1973:HIASTS}
\bysame, \emph{Homotopy invariant algebraic structures on topological spaces},
  Lecture Notes in Mathematics, vol. 347, Springer-Verlag, 1973.

\bibitem{Browder:1960:HOLS}
W.~Browder, \emph{Homology operations and loop spaces}, Illinois J. Math.
  \textbf{4} (1960), 347--357.

\bibitem{Cohen:preprint:CGTHTI}
F.~R. Cohen, \emph{Combinatorial group theory in homotopy theory, {I}},
  preprint.

\bibitem{Cohen:1987:CHT}
\bysame, \emph{A course in some aspects of classical homotopy theory}, Lecture
  Notes in Mathematics, vol. 1286, Springer, Berlin, 1987.

\bibitem{Cohen:1995:CGTH}
\bysame, \emph{On combinatorial group theory in homotopy}, Contemp. Math.
  \textbf{188} (1995), 57--63.

\bibitem{Coh-Lad-May:1976:HILS}
Frederick~R. Cohen, Thomas~J. Lada, and J.~Peter May, \emph{The homology of
  iterated loop spaces}, Lecture Notes in Mathematics, vol. 533,
  Springer-Verlag, 1976.

\bibitem{Hatcher:2002:AT}
Allen Hatcher, \emph{Algebraic topology}, Cambridge University Press,
  Cambridge, 2002.

\bibitem{LiWu:preprint:BGBBBHG}
Jingyan Li and Jie Wu, \emph{Braid groups, boundary {Brunnian} braids and
  homotopy groups}, preprint, 2007.

\bibitem{LodVal:AO}
Jean-Louis Loday and Bruno Vallette, \emph{Algebraic operads},
  http://www-irma.u-strasbg.fr/~loday/.

\bibitem{May:1972:GILS}
J.~P. May, \emph{The geometry of iterated loop spaces}, Lecture Notes in
  Mathematics, vol. 271, Springer-Verlag, 1972.

\bibitem{Neisendorfer:2010:AMUHT}
Joseph~A. Neisendorfer, \emph{Algebraic methods in unstable homotopy theory},
  New Mathematical Monographs, vol.~12, Cambridge University Press, Cambridge,
  2010.

\bibitem{Steenrod:1967:CCTS}
N.~E. Steenrod, \emph{A convenient category of topological spaces}, Mich. Math.
  J. \textbf{14} (1967), 133--152.

\bibitem{Whitehead:1978:EHT}
George~W. Whitehead, \emph{Elements of homotopy theory}, Graduate Texts in
  Mathematics, vol.~61, Springer-Verlag New York, 1978.

\bibitem{Zhang:2011:GOHT}
Wenbin Zhang, \emph{Group operads and homotopy theory}, preprint, 2011.

\end{thebibliography}
\bibliographystyle{amsplain}

\end{document}